\documentclass[12pt,reqno]{amsart}
\usepackage{amsmath}
\usepackage{amssymb}
\usepackage{amstext}
\usepackage{mathrsfs}
\usepackage{a4wide}
\usepackage{graphicx}
\usepackage{bbm}
\usepackage{verbatim}
\usepackage{bm}
\usepackage{graphicx}
\usepackage{tikz}
\usepackage[colorlinks,linkcolor=blue]{hyperref}
\usepackage{pgfplots}
\usepackage{tikz}
\usepackage {graphicx} 
\usepackage {epstopdf}
\usepackage{float}

\allowdisplaybreaks \numberwithin{equation}{section}
\usepackage{color}
\usepackage{cases}

\numberwithin{equation}{section}

\newtheorem{theorem}{Theorem}[section]

\newtheorem{lemma}[theorem]{Lemma}

\theoremstyle{definition}

\theoremstyle{remark}

\newcommand{\ep}{\varepsilon}
\newcommand{\Om}{\Omega}

\begin{document}

	\title
	[Regularization for point vortices on $\mathbb S^2$]{$C^1$ type regularization for point vortices on $\mathbb S^2$} 
	
	\author{Takashi Sakajo, Changjun Zou}
	
	\address{Department of Mathematics, Kyoto University, Kyoto, 606--8502, Japan}
	\email{sakajo@math.kyoto-u.ac.jp}
	\address{Department of Mathematics, Sichuan University, Chengdu, Sichuan, 610064, P.R. China}
	\email{zouchangjun@amss.ac.cn}
	
	\thanks{}

	\begin{abstract}
		We construct a series of classic vorticity solutions for incompressible Euler equation on $\mathbb S^2$, which constitute the $C^1$ type regularization for a general traveling point vortex system. The construction is accomplished by applying tangent mapping on $\mathbb S^2$ and Lyapunov--Schmidt reduction argument. Using the fixed-point theorem and a finite dimensional equation on vortex dynamics, we prove that the vortices are located near a nondegenerate critical point of Kirchhoff--Routh function. Moreover, in the tangent space at each vortex center, the scaled stream function is verified as a perturbation of the ground state for generalized plasma problem. Some other qualitative and quantitative estimates for the regularization series are also obtained in this paper. 
	\end{abstract}
	
	\maketitle{\small{\bf Keywords:}  Incompressible Euler equation on $\mathbb S^2$; $C^1$ type regularization for point vortices; Generalized plasma problem; Lyapunov--Schmidt reduction argument \\ 
		
		{\bf 2020 MSC} Primary: 76B47; Secondary: 76B03.}

	\section{Introduction}
	
	The motion of particles in an ideal fluid on the rotating unit sphere 
	$$\mathbb S^2:=\{ (x_1,x_2,x_3)\in\mathbb R^3 \mid x_1^2+x_2^2+x_3^3=1\},$$ 
	at uniform angular speed $\boldsymbol \gamma$ can be described by the following Euler equation of vorticity formulation (see \cite{Gre,Tay,Yuri}):
	\begin{align}\label{1-1}
		\begin{cases}
			\partial_t\omega+\boldsymbol{v}\cdot \nabla_{\mathbb S^2} (\omega-2\boldsymbol \gamma\cos\theta)=0, \quad \mathrm{on} \ \mathbb S^2,\\
			\ \boldsymbol{v}=\nabla^\perp_{\mathbb S^2}(-\Delta)^{-1}_{\mathbb S^2}\omega,\\
			\omega\big|_{t=0}=\omega_0,
		\end{cases}
	\end{align}
	where $\omega$ is the vorticity function, $\boldsymbol v$ is the velocity field, the term 
	$-2\boldsymbol\gamma\boldsymbol v\cdot \nabla_{\mathbb S^2}\cos\theta$ is the Coriolis force coming from the rotation of the sphere, and the second formula recovering $\boldsymbol v$ by $\omega$ is the Biot--Savart law on $\mathbb S^2$. In reality, \eqref{1-1} corresponds to the mathematical model for hurricane or circulation current on Earth.
	
	Generally speaking, the manifold $\mathbb S^2$ can be endowed with a smooth manifold structure described by the following two charts:
	\begin{align*}
		\mathcal{C}_1: (0,\pi)\times(0,2\pi)&\to \mathbb R^3\\
		(\theta,\varphi)&\mapsto (\sin\theta\cos\varphi, \sin\theta\sin\varphi,\cos\theta)\\
		\mathcal{C}_2: (0,\pi)\times(0,2\pi)&\to \mathbb R^3\\
		(\bar\theta,\bar\varphi)&\mapsto (-\sin\bar\theta\cos\bar\varphi, -\sin\bar\theta,-\sin\bar\theta\sin\bar\varphi).
	\end{align*}
	By assuming that the support of $\omega$ does not touch the south and north poles $ \{N, S\}$,  we can mainly work on the first chart $\mathcal{C}_1$, where the two variables $\theta$ and $\varphi$ are called the colatitude 
	and the longitude respectively. In this spherical coordinate, the Riemannian metric on $\mathbb S^2$ is given by
	\begin{equation*}
		\mathbf g_{\mathbb S^2}(\theta,\varphi)=d\theta^2+\sin^2\theta d\varphi^2.
	\end{equation*}
	For any point $\boldsymbol z = (\theta, \varphi) \in \mathbb S^2 \setminus \{N, S\}$ 
	on the sphere except for the two poles, an orthogonal
	basis of the tangent space $T_{\boldsymbol z}\mathbb S^2$ is given by
	\begin{equation*}
		\boldsymbol e_\theta=\partial_\theta \quad \mathrm{and} \quad \boldsymbol e_\varphi=\frac{\partial_\varphi}{\sin\theta},
	\end{equation*}
	where the classical identification between tangent vectors and directional differentiation is used. Denote $f(\boldsymbol z):\mathbb S^2\to \mathbb R$ as a function on $\mathbb S^2$. The integration for $f$ on the sphere can be written as 
	\begin{equation*}
		\int_{\mathbb S^2}f(\boldsymbol z)d\boldsymbol \sigma(\boldsymbol z)=\int_0^{2\pi}\int_0^{\pi}f(\theta,\varphi)\sin\theta d\theta d\varphi,
	\end{equation*}
	where
	\begin{equation*}
		d\boldsymbol\sigma=\sin\theta d\theta d\varphi
	\end{equation*}
	is the Riemannian volume.

	We shall get a step further to explicit the derivative notations on $\mathbb S^2$ appeared in \eqref{1-1}. Under the 
	spherical coordinate, the gradient of a scalar function $f$ on $\mathbb S^2$ is defined by
	\begin{equation*}
		\nabla_{\mathbb S^2} f(\theta,\varphi)=\partial_\theta f(\theta,\varphi)\boldsymbol e_\theta+\frac{\partial_\varphi f(\theta,\varphi)}{\sin\theta}\boldsymbol e_\varphi,
	\end{equation*}
	whose normal is 
	\begin{equation*}
		\nabla^{\perp}_{\mathbb{S}^2} f(\theta,\varphi)=J\nabla_{\mathbb{S}^2} f(\theta,\varphi),\quad\quad \mathrm{Mat}(J)=\left(
		\begin{array}{ccc}
			0 & 1 \\
			-1 & 0 \\
		\end{array}
		\right).
	\end{equation*}
	For a vector function $\boldsymbol f(\boldsymbol z) = (f_\theta(\theta, \varphi), f_\varphi(\theta, \varphi))$, the divergence is defined by
	\begin{equation*}
		\nabla_{\mathbb S^2} \cdot\boldsymbol f(\theta,\varphi)=\frac{1}{\sin\theta}\partial_\theta (\sin\theta f_\theta(\theta,\varphi)) +\frac{1}{\sin\theta}\partial_\varphi f_\varphi(\theta,\varphi).
	\end{equation*}
	Using the gradient and the divergence on $\mathbb S^2$, the Laplace--Beltrami operator $\Delta_{\mathbb S^2}$ can be expressed as 
	\begin{equation*}
		\Delta_{\mathbb S^2} f(\theta,\varphi)=\frac{1}{\sin\theta}\partial_\theta (\sin\theta\partial_\theta f(\theta,\varphi)) +\frac{1}{\sin^2\theta}\partial^2_\varphi f(\theta,\varphi),
	\end{equation*}
	where the north and south poles are two singular points.  For the inverse of Laplacian $(-\Delta_{\mathbb S^2})^{-1}$ in the Biot--Savart law of equation \eqref{1-1}, we have following Green representation
	\begin{equation*}
		(-\Delta)^{-1}_{\mathbb S^2}f=\int_{\mathbb S^2}G(\boldsymbol z,\boldsymbol z')d\boldsymbol\sigma(\boldsymbol z'),
	\end{equation*}
	where $G(\boldsymbol z,\boldsymbol z')$ is the Green function given by
	\begin{equation*}
		G(\theta,\varphi,\theta',\varphi')=-\frac{1}{4\pi}\ln\big(1-\cos\theta\cos\theta'-\sin\theta\sin\theta'\cos(\varphi-\varphi')\big)+\frac{\ln2}{4\pi}.
	\end{equation*}
	It follows from
	\begin{equation*}
		1-\cos\theta\cos\theta'-\sin\theta\sin\theta'\cos(\varphi-\varphi')=2\left[\sin^2\left(\frac{\theta-\theta'}{2}\right)+\sin\theta\sin\theta'\sin^2\left(\frac{\varphi-\varphi'}{2}\right)\right]
	\end{equation*}
	that the singular part in $G(\boldsymbol z,\boldsymbol z')$ is similar to the fundamental solution $-\frac{1}{2\pi}\ln\frac{1}{|\boldsymbol x|}$ for $-\Delta$ in $\mathbb R^2$, which is defined as
	\begin{equation*}
		\Gamma(\theta,\varphi,\theta',\varphi')=-\frac{1}{4\pi}\ln\big[(\theta-\theta')^2+(\varphi-\varphi')^2 \sin^2\theta\big].
	\end{equation*}
	Then we can split $G(\boldsymbol z,\boldsymbol z')$ into two parts
	\begin{equation}\label{1-2}
		\begin{split}
			G(\theta,\phi,\theta',\phi')&=\Gamma(\theta,\varphi,\theta',\varphi')+\left[G-\Gamma(\theta,\varphi,\theta',\varphi')\right]\\
			&=\Gamma(\theta,\varphi,\theta',\varphi')+H(\theta,\varphi,\theta',\varphi'),
		\end{split}
	\end{equation}
	where $H(\theta,\varphi,\theta',\varphi')\in C^1(B_\delta(\boldsymbol z)\times B_\delta(\boldsymbol z))$ for $\boldsymbol z\in \mathbb S^2\setminus \{N, S\}$ is the remaining regular part, and $B_\delta(\boldsymbol z)$ is a spherical cap around $\boldsymbol z$ with radius $\delta$.
	
	\smallskip
	
	For the global well-posedness of \eqref{1-1} with $\omega_0\in L^1\cap L^\infty$, we refer to \cite{Tay} following the approach of Yudovichi \cite{Yud}. Besides the original Cauchy problem, since the fluid patterns, in reality, can often be regarded as a perturbation near the relative equilibrium states, these equilibria or global solutions play an important role in mathematicians' and physicists' understanding of the behavior of the fluid. For example, stratospheric planetary flows for different planets in our solar system can be modeled by the bifurcation of zonal solutions \cite{CG}, and periodic vortex shedding along the equator on Jupiter is very similar to the von K\'arm\'an vortex street on $\mathbb S^2$ \cite{SZ}.
	
	For the planer incompressible Euler equation,
	\begin{align*}
		\begin{cases}
			\partial_t\omega+\boldsymbol{v}\cdot \nabla \omega=0, \quad \mathrm{in} \ \mathbb R^2,\\
			\ \boldsymbol{v}=\nabla^\perp(-\Delta)^{-1}\omega,\\
			\omega\big|_{t=0}=\omega_0,
		\end{cases}
	\end{align*}
	various methods have established different relative equilibrium states. The first non-trivial example is the famous Kirchhoff ellipse \cite{Kir}, where the vorticity
	is the characteristic function of an ellipse with semi-axes $a,b$ and a uniform angular velocity $ab/(a+b)^2$. It is obvious that the Kirchhoff ellipse is $2$-fold symmetric, and this kind of uniform rotating solutions are now called V-states \cite{Deem} of $N$-fold symmetry. Burbea \cite{Burb} gives the first attempt to construct V-states bifurcated from a unit disk for $N\ge3$, which is made rigorous by Hmidi et al. \cite{HMV}. Their approach is the contour dynamics by studying the functional equation for the vorticity boundary. Since then, different kinds of V-states have been constructed through a delicate analysis of the spectrum for the linearized problem. For instance, Hoz et al.\cite{de2} obtain the doubly-connected V-states near an annulus; Castro et al. consider the case of bifurcation from Kirchhoff ellipse in \cite{Cas4}, and extend the theory to smooth solutions in \cite{Cas3}; Hassainia et al.~\cite{Has0} obtain a global bifurcation curve for rotating vortex patches. Besides the V-state solutions, another typical class of equilibrium states of highly concentrated vortices are also of special interest. This type of solution is usually constructed by regularization of the point vortex system, and the first rigorous construction for concentrated vortex patches was due to Turkington \cite{T1} through a dual variational principle of Arnol$'$d \cite{Ar2}, which is developed in \cite{CWZ,Go} for more general vorticity profile function and gSQG equations. Then in \cite{Ao,CPY,SV}, the authors transform the regularization process into a semi-linear elliptic problem and apply the Lyapunov--Schmidt reduction or Mountain path method to deal with the construction. Hmidi and Mateu~\cite{HM} also apply the dynamics to construct co-rotating and counter-rotating patches near a vortex pair. 
	
    Compared with the incompressible Euler equation on $\mathbb R^2$, there is relatively less work on the global solutions for  \eqref{1-1}, where longitude independent $\omega(\theta,\varphi)=\psi(\theta)$ constitutes the most common trivial steady states known as zonal solutions. Constantin and Germain~\cite{CG} consider local and global bifurcation of non-zonal solutions to \eqref{1-1} from Rossby--Haurwitz waves. They prove the stability in $H^2(\mathbb S^2)$ of degree $2$ case as well as the instability in $H^2(\mathbb S^2)$ of 
    more general non-zonal Rossby--Haurwitz type solutions. They also show that any steady solution $\omega$ to \eqref{1-1} of the form
    $$\omega=\mathcal{F}\left((-\Delta)^{-1}_{\mathbb S^2}\omega\right)-2\boldsymbol \gamma\cos\theta,$$
    with $\mathcal{F}:\mathbb R\to \mathbb R$ a smooth function satisfying $\mathcal{F}'>-6$ must be zonal (modulo rotation), and stable in $H^2(\mathbb S^2)$ provided an additional constraint $\mathcal{F}' < 0$. (The validity for functional relationship $\mathcal{F}$ as a sufficient condition for $\omega$ being a steady solution can be verified as a analogy of \eqref{1-3} later, and the constant $-6$ corresponds to the second eigenvalues of the Laplace--Beltrami operator.) Then Garc\'ia et al.~\cite{GHR} developed the idea in \cite{HMV}, and constructed $k$-fold symmetric vortex cap solutions by using contour dynamics equations and bifurcation tool. On the other hand, the point vortex system on $\mathbb S^2$ is an interesting topic for the forecast of tornadoes and hurricanes, whose dynamic is studied in \cite{New0, New, Wang}. Very recently, the authors of this paper \cite{SZ} give the first attempt at regularization for point vortices on $\mathbb S^2$, and construct two types of von K\'arm\'an vortex patch street on $\mathbb S^2$.
    
	It is noteworthy that in the previous work \cite{GHR,SZ}, the vortices are given by characteristic function for several subdomains of $\mathbb S^2$, and hence the solution $\omega$ is discontinuous on the boundary of vortex sets. According to the regularity theory of elliptic operator, the stream function $\psi=(-\Delta)^{-1}_{\mathbb S^2}\omega$ is in $C^{1,1}$ and velocity field $\boldsymbol v$ is only Lipschitz smooth. For $\omega$ in those situations, we see that it only gives a weak solution to \eqref{1-1}. In this paper, we regularize the point vortices on $\mathbb S^2$ in a smoother type, and obtain a family of $C^1$ relative equilibrium states $\omega$, which constitute the classic global solutions to \eqref{1-1}. Moreover, we prove that the stream function $\psi$ after scaling tend to the ground state of the generalized plasma problem, and the boundary of each vortex area is a $C^2$ convex curve due to the nice regularity.
	
	To begin with, let us consider the traveling wave solutions to \eqref{1-1} with sphere rotation speed $\boldsymbol \gamma=0$, 
	which is obtained by a translation acting on initial data $\omega_0$, namely, 
	\begin{equation*}
		\omega(t,\theta,\varphi)=\omega_0(\theta, \varphi+Wt)
	\end{equation*}
	with $W$ the traveling speed along the latitude. By substituting it into \eqref{1-1} and letting $\psi=(-\Delta)^{-1}_{\mathbb S^2}\omega_0$ be the stream function, we find that the first equation of \eqref{1-1} becomes
	\begin{equation*}
		\nabla_{\mathbb{S}^2}^\perp(\psi+W\cos\theta)\cdot \nabla_{\mathbb S^2} \omega_0=0,
	\end{equation*}
	which means that $\psi+W\cos\theta$ and $\omega_0$ are functional dependent. Hence if we let $F(\tau): \mathbb R\to \mathbb R$ be a suitable vorticity profile function, by imposing the functional relationship
	\begin{equation}\label{1-3}
		(-\Delta_{\mathbb S^2})\psi=\omega_0=F(\psi+W\cos\theta),
	\end{equation}
	then $\omega_0$ gives a traveling wave solution to \eqref{1-1} on the stationary sphere, that is to say, $\omega_0$ satisfies
		\begin{align*}
		\begin{cases}
			(\boldsymbol{v}_0+(0,W))\cdot \nabla_{\mathbb S^2} \omega_0=0,\quad \mathrm{on} \ \mathbb S^2,\\
			\ \boldsymbol{v}_0=\nabla^\perp_{\mathbb S^2}(-\Delta)^{-1}_{\mathbb S^2}\omega_0.\\
		\end{cases}
	\end{align*}
	In \cite{SZ}, the profile function is of the type $F(\tau)=\boldsymbol 1_{\{\tau>0\}}$, while in this paper we will let $F(\tau)=\tau_+^\gamma$ with $\gamma\ge 1$ so that $F\in C^1(\mathbb R)$ ($F$ become smoother as $\gamma$ increase). Thus we can apply the regularity theory of Laplacian and bootstrap to obtain $\omega\in C^1$ and $\psi\in C^3$.
	
	Since we are to regularize a traveling point vortex system on $\mathbb S^2$, the limit function for vorticity $\omega_0$ should be 
	\begin{equation}\label{1-4}
		\omega^*(\boldsymbol z)=\sum\limits_{m=1}^{j}\kappa_m^+\boldsymbol \delta_{\boldsymbol z_m^+}-\sum\limits_{n=1}^{k}\kappa_n^-\boldsymbol \delta_{\boldsymbol z_n^-}.
	\end{equation}
	On the other hand, given the divergence theorem, we can deduce the Gauss constraint for \eqref{1-1} as
	\begin{equation*}
		\int_{\mathbb S^2}\omega(t,\boldsymbol z)d\boldsymbol \sigma(\boldsymbol z)=0,\quad \forall \, t\in[0,+\infty).
	\end{equation*}
	Thus it is natural to assume 
	$$\sum\limits_{m=1}^{j}\kappa_m^+=\sum\limits_{n=1}^{k}\kappa_n^-.$$
	Actually, the point vortices on $\mathbb S^2$ are transported by the velocity field generated by other vortices, whose dynamics are dominated by a Hamiltonian 
	system induced by the corresponding Kirchhoff--Routh function. For this topic, readers can refer to \cite{Lin} for the case on $\mathbb R^2$, and \cite{Drit, Sakajo-Torus, VS-flat-torus} for the cases on $\mathbb S^2$, the curved torus $\mathbb{T}_{R,r}$ and the flat torus $\mathbb{T}^2$ separately. 
	Since the point vortex system in \eqref{1-4} is a relative equilibrium state, the coordinates $(\boldsymbol z_1^+,\cdots,\boldsymbol z_j^+, \boldsymbol z_1^-,\cdots, \boldsymbol z_k^-)$ with $\boldsymbol z_m^+=(\phi_m^+,\theta_m^+)$, $\boldsymbol z_n^-=(\phi_n^-,\theta_n^-)$ ($\theta_m^+,\theta_n^-\neq 0,\pi$) should be on a stagnation point of velocity field, namely, a critical point of the Kirchhoff--Routh function
	\begin{equation}\label{1-5}
		\begin{split}
			\mathcal K_{k+j}&(\boldsymbol z_1^+,\cdots,\boldsymbol z_j^+,\boldsymbol z_1^-,\cdots,\boldsymbol z_k^-)\\
			&=\frac{1}{2}\sum_{m,l=1,m\neq l}^j\kappa_m^+\kappa_l^+G(\boldsymbol z_m^+,\boldsymbol z_l^+)+\frac{1}{2}\sum_{l,n=1,l\neq n}^k\kappa_l^-\kappa_n^-G(\boldsymbol z_l^-,\boldsymbol z_n^-)\\
			&\quad+\frac{1}{2}\sum_{m=1}^k(\kappa_m^+)^2H(\boldsymbol z_m^+,\boldsymbol z_m^+)+\frac{1}{2}\sum_{n=1}^j(\kappa_n^-)^2H(\boldsymbol z_n^-,\boldsymbol z_n^-)\\
			&\quad-\sum_{m=1}^j\sum_{n=1}^k \kappa_m^+\kappa_n^- G(\boldsymbol z_m^+,\boldsymbol z_n^-)-W\sum_{m=1}^j\kappa_m^+\cos \theta_m^++W\sum_{n=1}^k\kappa_n^-\cos \theta_n^-.
		\end{split}
	\end{equation}
	Now we construct a series of vortex solutions $\omega_\ep$ to regularize $\omega^*(\boldsymbol z)$. In view of the functional relationship $F$ derived in \eqref{1-3}, $\psi_\varepsilon=(-\Delta)^{-1}_{\mathbb S^2}\omega_\ep$ can be chosen to satisfy the equation
	\begin{equation}\label{1-6}
		(-\Delta_{\mathbb S^2})\psi_\varepsilon=\frac{1}{\varepsilon^2}\sum_{m=1}^j\boldsymbol1_{B_\delta(\boldsymbol z_m^+)}(\psi_\varepsilon+W\cos\theta-\mu_{m,\varepsilon}^+)_+^\gamma-\frac{1}{\varepsilon^2}\sum_{n=1}^k\boldsymbol1_{B_\delta(\boldsymbol z_n^-)}(-\psi_\varepsilon-W\cos\theta>\mu_{n,\varepsilon}^-)_+^\gamma,
	\end{equation}
	where the power index $\gamma\ge1$, the radius $\delta>0$ is a small constant, $\mu_{m,\varepsilon}^+,\mu_{n,\varepsilon}^-$ are flux constants to be prescribed, and
	$$\{\psi_\varepsilon+W\cos\theta>\mu_{m,\varepsilon}^+\}=\Omega_{m,\varepsilon}^+, \quad\quad \{-\psi_\varepsilon-W\cos\theta>\mu_{n,\varepsilon}^-\}=\Omega_{n,\varepsilon}^-$$ 
	are the positive and negative vortex sets. We have the following theorem on the regularization of \eqref{1-4} as solutions to \eqref{1-5} with $\gamma>1$, whose limit stream function is the solution to generalized plasma problem $-\Delta u=u_+^\gamma$.
	
	\begin{theorem}\label{thm1}
		Suppose that $\gamma>1$. Then for any nondegenerate critical point $(\boldsymbol z_1^+,\cdots,\boldsymbol z_j^+, \boldsymbol z_1^-,\cdots, \boldsymbol z_k^-)$ of Kirchhoff--Routh function $\mathcal K_{k+j}$ defined by \eqref{1-5}, there exists a small $\varepsilon_0>0$, such that for each $\varepsilon\in (0,\varepsilon_0]$, \eqref{1-6} has a solution $\psi_\varepsilon$. Let $\omega_\varepsilon=(-\Delta_{\mathbb S^2})\psi_\varepsilon$ be the vorticity function. We have the following asymptotic estimates:
		\begin{itemize}
			\item[(i)] One has
			$$\omega_\varepsilon\rightharpoonup \sum\limits_{m=1}^{j}\kappa_m^+\boldsymbol \delta_{\boldsymbol z_m^+}-\sum\limits_{n=1}^{k}\kappa_n^-\boldsymbol \delta_{\boldsymbol z_n^-} \quad \quad \mathrm{as} \ \varepsilon\to 0^+,$$
			where the convergence is in the sense of measures, and $\kappa_l^\pm$ is the circulation for each point vortex.
			\item [(ii)] Let $w_\gamma(\boldsymbol y)$ be the unique radial solution of
			\begin{equation*}
				-\Delta w_\gamma(\boldsymbol y)=w_\gamma^\gamma, \quad w_\gamma\in H_0^1( B_1(\boldsymbol 0)), \quad w_\gamma(\boldsymbol y)>0 \ \mathrm{in} \ B_1(\boldsymbol 0),
			\end{equation*}
			and $s_{l,\varepsilon}^\pm$ satisfy $\left(\frac{\ep}{s_{l,\ep}^\pm}\right)^{\frac{2}{\gamma-1}}w'_\gamma(1)=\frac{\kappa_l^\pm}{2\pi}\frac{|\ln\varepsilon|}{|\ln s_{l,\varepsilon}^\pm|}$. Then as $\ep\to 0^+$, one has
			$$\left(\frac{s_{l,\varepsilon}^\pm}{\ep}\right)^{\frac{2}{\gamma-1}}\psi_\varepsilon(\mathcal A^\pm_{l,\ep}\boldsymbol y+\boldsymbol z^\pm_{l,\ep})\to w_\gamma(\boldsymbol y) \quad \mathrm{in} \ C^3(B_2(\boldsymbol 0)), $$
		    $$\left(\frac{s_{l,\varepsilon}^\pm}{\ep}\right)^{\frac{2\gamma}{\gamma-1}}\omega_\varepsilon(\mathcal A^\pm_{l,\ep}\boldsymbol y+\boldsymbol z^\pm_{l,\ep})\to \frac{(w_\gamma(\boldsymbol y))^\gamma_+}{\ep^2} \quad\mathrm{in} \ C^1(B_2(\boldsymbol 0)),$$
			where the matrix of tangent mapping $\mathcal A^\pm_{l,\ep}$ is 
			\begin{equation*}
				\mathrm{Mat}(\mathcal A^\pm_{l,\ep})=\left(
				\begin{array}{ccc}
					s^{\pm}_{l,\ep} & 0             \\
					0 & s^{\pm}_{l,\ep} \sin^{-1} \theta^\pm_{l,\ep} 
				\end{array}
				\right),
			\end{equation*}
			and
			$$\left(\boldsymbol z_{1,\varepsilon}^+,\cdots,\boldsymbol z_{j,\varepsilon}^+, \boldsymbol z_{1,\varepsilon}^-,\cdots, \boldsymbol z_{k,\varepsilon}^-\right)\to \left(\boldsymbol z_1^+,\cdots,\boldsymbol z_j^+, \boldsymbol z_1^-,\cdots, \boldsymbol z_k^-\right).$$
			\item[(iii)] The boundaries of level sets $\Omega_{m,\varepsilon}^+$, $\Omega_{n,\varepsilon}^-$ are $C^2$ smooth convex curves, which are parameterized as
			$$\partial \Omega_{m,\varepsilon}^+=\left\{\boldsymbol z_{m,\varepsilon}^++\left[\left(\frac{2\pi w'_\gamma(1)}{\kappa_m^+}\right)^{\frac{\gamma-1}{2}}\varepsilon+o(\varepsilon)\right](\cos\xi,\sin^{-1}\theta_{m,\ep}^+\sin\xi) \, \bigg| \, \xi\in[0, 2\pi) \right\},$$
			and
			$$\partial \Omega_{n,\varepsilon}^-=\left\{\boldsymbol z_{n,\varepsilon}^-+\left[\left(\frac{2\pi w'_\gamma(1)}{\kappa_n^-}\right)^{\frac{\gamma-1}{2}}\varepsilon+o(\varepsilon)\right](\cos\xi,\sin^{-1}\theta_{n,\ep}^-\sin\xi) \, \bigg| \, \xi\in[0, 2\pi) \right\}.$$
		\end{itemize}
	\end{theorem}
	
	Since $(-\Delta_{\mathbb S^2})\cos\theta=2\cos\theta$, it is proved in Lemma 1.1 of \cite{SZ} that once we construct a localized vorticity solution $\omega_\ep$ on stationary sphere ($\boldsymbol \gamma=0$), then we obtain a class of non-localized solutions with different sphere rotation speed $\boldsymbol \gamma$. In fact, if we let $\tilde\omega_\ep=\omega_\varepsilon+2\boldsymbol \gamma\cos\theta$ and $\tilde{\boldsymbol v}=\nabla^\perp_{\mathbb S^2}\psi_\varepsilon+(0,\boldsymbol \gamma\sin\theta)$, then $\tilde\omega_\ep$ gives a traveling wave solution at the speed $W_{\boldsymbol \gamma}=W-\boldsymbol \gamma$.
	
	\smallskip
	
	For the case $\gamma=1$, the limit equation of \eqref{1-6} corresponds to the classic plasma problem $-\Delta u=u_+$.  Since the proof and asymptotic estimate are different from the case $\gamma>1$, we state the result separately.
    \begin{theorem}\label{thm2}
    	Suppose that $\gamma=1$. Then for any nondegenerate critical point $(\boldsymbol z_1^+,\cdots,\boldsymbol z_j^+, \boldsymbol z_1^-,\cdots, \boldsymbol z_k^-)$ of Kirchhoff--Routh function $\mathcal K_{k+j}$ defined by \eqref{1-5}, there exists a small $\varepsilon_0>0$, such that for each $\varepsilon\in (0,\varepsilon_0]$, \eqref{1-6} has a solution $\psi_\varepsilon$. Let $\omega_\varepsilon=(-\Delta_{\mathbb S^2})\psi_\varepsilon$ be the vorticity function. We have the following asymptotic estimates:
    	\begin{itemize}
    		\item[(i)] One has
    		$$\omega_\varepsilon\rightharpoonup \sum\limits_{m=1}^{j}\kappa_m^+\boldsymbol \delta_{\boldsymbol z_m^+}-\sum\limits_{n=1}^{k}\kappa_n^-\boldsymbol \delta_{\boldsymbol z_n^-} \quad \quad \mathrm{as} \ \varepsilon\to 0^+,$$
    		where the convergence is in the sense of measures, and $\kappa_l^\pm$ is the circulation for each point vortex.
    		\item [(ii)] Let $\tau>0$ be the constant such that $1$ is the first eigenvalue of $-\Delta$ in $B_\tau(\boldsymbol 0)$ with the zero Dirichlet boundary condition, and the radial decreasing function $w_1(\boldsymbol y)> 0$ be the first eigenfunction for $-\Delta$ in $B_{\tau}(\boldsymbol 0)$ with $w_1(\boldsymbol 0)=1$. Then as $\ep\to 0^+$, one  has
    		$$\frac{\kappa}{\kappa_l^\pm}\cdot\psi_\varepsilon(\mathcal A^\pm_{l,\ep}\boldsymbol y+\boldsymbol z^\pm_{l,\ep})\to w_1(\boldsymbol y) \quad \mathrm{in} \ C^3(B_{2\tau}(\boldsymbol 0)),$$
    		$$\frac{\kappa}{\kappa_l^\pm}\cdot\omega_\varepsilon(\mathcal A^\pm_{l,\ep}\boldsymbol y+\boldsymbol z^\pm_{l,\ep})\to \frac{(w_1(\boldsymbol y))_+}{\ep^2} \quad \mathrm{in} \ C^1(B_{2\tau}(\boldsymbol 0)),$$
    		where $\kappa=\int_{\mathbb R^2} (w_1(\boldsymbol y))_+d\boldsymbol y$, the matrix of tangent mapping $\mathcal A^\pm_{l,\ep}$ is 
    		\begin{equation*}
    			\mathrm{Mat}(\mathcal A^\pm_{l,\ep})=\left(
    			\begin{array}{ccc}
    				\ep & 0             \\
    				0 & \ep \sin^{-1} \theta^\pm_{l,\ep} 
    			\end{array}
    			\right),
    		\end{equation*}
    		and
    		$$\left(\boldsymbol z_{1,\varepsilon}^+,\cdots,\boldsymbol z_{j,\varepsilon}^+, \boldsymbol z_{1,\varepsilon}^-,\cdots, \boldsymbol z_{k,\varepsilon}^-\right)\to \left(\boldsymbol z_1^+,\cdots,\boldsymbol z_j^+, \boldsymbol z_1^-,\cdots, \boldsymbol z_k^-\right).$$
    		\item[(iii)] The boundaries of level sets $\Omega_{m,\varepsilon}^+$, $\Omega_{n,\varepsilon}^-$, are $C^2$ smooth convex curves, which are parameterized as
    		$$\partial \Omega_{m,\varepsilon}^+=\left\{\boldsymbol z_{m,\varepsilon}^++\left[\tau\varepsilon+o(\varepsilon)\right](\cos\xi,\sin^{-1}\theta_{m,\ep}^+\sin\xi) \, \big| \, \xi\in[0, 2\pi) \right\},$$
    		and
    		$$\partial \Omega_{n,\varepsilon}^-=\left\{\boldsymbol z_{n,\varepsilon}^-+\left[\tau\varepsilon+o(\varepsilon)\right](\cos\xi,\sin^{-1}\theta_{n,\ep}^-\sin\xi) \, \big| \, \xi\in[0, 2\pi) \right\}.$$
    	\end{itemize}
    \end{theorem}

	This paper is organized as follows. In Section \ref{sec2}, we prove the existence of solutions to \eqref{1-6} with $\gamma>1$ by Lyapunov--Schmidt reduction argument. We first give a series of approximate solutions and derive the linearized equation. Then we consider the projection problem and solve it by the contraction mapping theorem. The last step of our construction is to choose suitable locations for vortices so that the projection operator equals the identity. In Section \ref{sec3}, we complete the discussion by giving the solution series to \eqref{1-6} for the case $\gamma=1$.

	\bigskip
	
	\section{Regularization for the case $\gamma>1$}\label{sec2}
	
	Since the approximate solution for case $\gamma=1$ is very different from that for $\gamma>1$, we separate the discussion for $\gamma=1$ in Section \ref{sec3} and start with $\gamma>1$.
	
	\subsection{The approximate solution}
	
	Note that the stream function for point vortex system \eqref{1-4} is
	$$\psi^*=\sum_{m=1}^j\kappa^+_mG(\boldsymbol z,\boldsymbol z^+_{m,\varepsilon})-\sum_{n=1}^k\kappa^-_nG(\boldsymbol z,\boldsymbol z^-_{n,\varepsilon}),$$
	which has a log-singularity at each vortex point $\boldsymbol z^\pm_{l,\varepsilon}$. Our first step is to define a series of regular functions to substitute $\kappa^\pm_lG(\boldsymbol z,\boldsymbol z^\pm_{l,\varepsilon})$. Since the support of vorticity function $\omega_\ep$ shrink to $j+k$ points on $\mathbb S^2$ as $\ep\to0^+$, the limit function for $\psi_\ep$ at $\boldsymbol z^\pm_{l,\varepsilon}$ should satisfy $-\Delta w=w_+^\gamma$ in the tangent space $T_{\boldsymbol z_l^\pm}{\mathbb S^2}$, where $\Delta$ is standard Laplacian in $\mathbb R^2$. Inspired by this observation, for $\gamma>1$, we let $w_\gamma(\boldsymbol y)$ be the unique radial solution of the generalized plasma problem
	\begin{equation*}
		-\Delta w_\gamma(\boldsymbol y)=w_\gamma^\gamma, \quad w_\gamma\in H_0^1(B_1(\boldsymbol 0)), \quad w_\gamma(\boldsymbol y)>0 \ \mathrm{in} \ B_1(\boldsymbol 0),
	\end{equation*}
	which is also known as the ground state solution, and define the tangent mapping $A: (\theta,\varphi)\to (x_1,x_2)$ from $\mathbb S^2$ to $T_{\boldsymbol z}{\mathbb S^2}$ by the matrix
	\begin{equation*}
		\mathrm{Mat}(A)=\left(
		\begin{array}{ccc}
			1 & 0             \\
			0 & \sin\theta
		\end{array}
		\right).
	\end{equation*}
	Then the singular part $\kappa^\pm_l\Gamma(\boldsymbol z,\boldsymbol z^\pm_{l,\varepsilon})$ can be smoothed by
	\begin{equation*}
		V^\pm_{l,\varepsilon}(\boldsymbol z)=\left\{
		\begin{array}{lll}
			\frac{\kappa_l^\pm}{2\pi}\ln\frac{1}{\varepsilon}+\left(\frac{\ep}{s_{l,\varepsilon}^\pm}\right)^{\frac{2}{\gamma-1}}w_\gamma\left(\frac{A(\boldsymbol z-\boldsymbol z_{l,\varepsilon}^\pm)}{s_{l,\varepsilon}^\pm}\right), \ \ \ &\mathrm{if} \ |A(\boldsymbol z-\boldsymbol z^\pm_{l,\varepsilon})|\le s_{l,\varepsilon}^\pm,\\
			\frac{\kappa_l^\pm}{2\pi}\frac{|\ln\varepsilon|}{|\ln s_{l,\varepsilon}^\pm|}\cdot\ln|A(\boldsymbol z-\boldsymbol z^\pm_{l,\varepsilon})|,&\mathrm{if} \ |A(\boldsymbol z-\boldsymbol z^\pm_{l,\varepsilon})|\ge s_{l,\varepsilon}^\pm,
		\end{array}
		\right.
	\end{equation*}
	where we impose
	\begin{equation}\label{2-1}
		s_{l,\ep}^\pm\beta_{l,\ep}^\pm=\left(\frac{\ep}{s_{l,\ep}^\pm}\right)^{\frac{2}{\gamma-1}}w'_\gamma(1)=\frac{\kappa_l^\pm}{2\pi}\frac{|\ln\varepsilon|}{|\ln s_{l,\varepsilon}^\pm|}.
	\end{equation}
	Here $\beta_{l,\ep}^\pm$ is the value of $-\partial_\theta V_{l,\varepsilon}^\pm$ at $(s_\varepsilon+\theta_{l,\ep}^\pm,\varphi_{l,\ep}^\pm)$ to make $V^\pm_{l,\varepsilon}$ a $C^1$ function. By the formulation of $V_{i,\varepsilon}^\pm$, it holds the following integration equation.
	\begin{equation*}
		V^\pm_{l,\varepsilon}(\boldsymbol z)=\frac{1}{\varepsilon^2}\int_{\{V^\pm_{l,\varepsilon}(\boldsymbol z)>\frac{\kappa_l^\pm}{2\pi}\ln\frac{1}{\varepsilon}\}} \Gamma(\theta,\varphi,\theta',\varphi')\left(V^\pm_{l,\varepsilon}-\frac{\kappa_l^\pm}{2\pi}\ln\frac{1}{\varepsilon}\right)_+^\gamma d\boldsymbol \sigma(\boldsymbol z').
	\end{equation*}
	For the remaining regular part $\kappa^\pm_l H(\boldsymbol z,\boldsymbol z^\pm_{l,\varepsilon})$, we let
	\begin{equation*}
		R^\pm_{l,\varepsilon}(\boldsymbol z)=\frac{1}{\varepsilon^2}\int_{\{V^\pm_{l,\varepsilon}(\boldsymbol z)>\frac{\kappa_l^\pm}{2\pi}\ln\frac{1}{\varepsilon}\}} H(\theta,\varphi,\theta',\varphi')\left(V^\pm_{l,\varepsilon}-\frac{\kappa_l^\pm}{2\pi}\ln\frac{1}{\varepsilon}\right)_+^\gamma d\boldsymbol \sigma(\boldsymbol z').
	\end{equation*} 
	According to the definition of $V^\pm_{i,\varepsilon}(\boldsymbol z)$ and $R^\pm_{i,\varepsilon}(\boldsymbol z)$, one can verify that 
	\begin{equation*}
		(-\Delta_{\mathbb S^2})\left(V^\pm_{l,\varepsilon}+R^\pm_{l,\varepsilon}\right)=\frac{1}{\varepsilon^2}\cdot\left(V^\pm_{l,\varepsilon}(\boldsymbol z)-\frac{\kappa_l^\pm}{2\pi}\ln\frac{1}{\varepsilon}\right)_+^\gamma.
	\end{equation*}
	Thus we can split $\psi_\varepsilon(\boldsymbol z)$ as
	\begin{align*}
		\psi_\varepsilon(\boldsymbol z)&=\sum\limits_{m=1}^jV^+_{m,\varepsilon}+\sum\limits_{m=1}^jR^+_{m,\varepsilon}-\sum\limits_{n=1}^kV^-_{n,\varepsilon}-\sum\limits_{n=1}^kR^-_{n,\varepsilon}+\phi_\varepsilon\\
		&:= \Psi_\varepsilon+\phi_\varepsilon 
	\end{align*}
	with $\Psi_\varepsilon(\boldsymbol z)$ the approximate solution, and $\phi_\varepsilon(\boldsymbol z)$ a small error term. To localize each vortex, we denote $B_\delta(\boldsymbol z_i^\pm)\subset \mathbb S^2$ as the spherical cap region with radius $\delta$. By substituting $\psi_\ep$ with the above decomposition into \eqref{1-6}, the equation can be linearized as
	\begin{align*}
		0&=-\varepsilon^2\Delta_{\mathbb S^2}\left(\sum\limits_{m=1}^jV^+_{i,\varepsilon}+\sum\limits_{m=1}^jR^+_{i,\varepsilon}-\sum\limits_{n=1}^kV^-_{i,\varepsilon}-\sum\limits_{n=1}^kR^-_{i,\varepsilon}+\phi_\varepsilon\right)\\
		&\quad-\sum\limits_{m=1}^j\boldsymbol1_{B_\delta(\boldsymbol z_m^+)}\left(\psi_\varepsilon+W_\varepsilon\cos\theta-\mu_{m,\varepsilon}^+\right)_+^\gamma+\sum\limits_{n=1}^k\boldsymbol1_{ B_\delta(\boldsymbol z_n^-)}\left(-\psi_\varepsilon-W_\varepsilon\cos\theta-\mu_{n,\varepsilon}^-\right)_+^\gamma\\
		&=\sum\limits_{m=1}^j\left(-\varepsilon^2\Delta_{\mathbb S^2}\left(V^+_{m,\varepsilon}+R^+_{m,\varepsilon}\right)-\left(V^+_{m,\varepsilon}(\boldsymbol z)-\frac{\kappa_m^+}{2\pi}\ln\frac{1}{\varepsilon}\right)_+^\gamma\right)\\
		&\quad -\sum\limits_{n=1}^k\left(-\varepsilon^2\Delta_{\mathbb S^2}\left(V^-_{n,\varepsilon}+R^-_{n,\varepsilon}\right)-\left(V^-_{n,\varepsilon}(\boldsymbol z)-\frac{\kappa_n^-}{2\pi}\ln\frac{1}{\varepsilon}\right)_+^\gamma\right)\\
		& \ \ \ +\left(-\varepsilon^2\Delta_{\mathbb S^2}\phi_\varepsilon-\sum\limits_{m=1}^j \gamma\left(V^+_{m,\varepsilon}(\boldsymbol z)-\frac{\kappa_m^+}{2\pi}\ln\frac{1}{\varepsilon}\right)_+^{\gamma-1}\phi_\varepsilon-\sum\limits_{n=1}^k\gamma\left(V^-_{n,\varepsilon}(\boldsymbol z)-\frac{\kappa_n^-}{2\pi}\ln\frac{1}{\varepsilon}\right)_+^{\gamma-1}\phi_\varepsilon\right)\\
		& \ \ \ -\sum\limits_{m=1}^j\left(\boldsymbol1_{B_\delta(\boldsymbol z_m^+)}\left(\psi_\varepsilon+W_\varepsilon\cos\theta-\mu_{m,\varepsilon}^+\right)_+^\gamma-\left(V^+_{m,\varepsilon}(\boldsymbol z)-\frac{\kappa_m^+}{2\pi}\ln\frac{1}{\varepsilon}\right)_+^\gamma\right.\\
		&\left. \quad\quad-\gamma\left(V^+_{m,\varepsilon}(\boldsymbol z)-\frac{\kappa_m^+}{2\pi}\ln\frac{1}{\varepsilon}\right)_+^{\gamma-1}\phi_\varepsilon\right)\\
		& \ \ \ +\sum\limits_{n=1}^k\left(\boldsymbol1_{ B_\delta(\boldsymbol z_n^-)}\left(-\psi_\varepsilon-W_\varepsilon\cos\theta-\mu_{n,\varepsilon}^-\right)_+^\gamma-\left(V^-_{n,\varepsilon}(\boldsymbol z)-\frac{\kappa_n^-}{2\pi}\ln\frac{1}{\varepsilon}\right)_+^\gamma\right.\\
		&\left.\quad\quad+\gamma\left(V^-_{n,\varepsilon}(\boldsymbol z)-\frac{\kappa_n^-}{2\pi}\ln\frac{1}{\varepsilon}\right)_+^{\gamma-1}\phi_\varepsilon\right)\\
		&=\varepsilon^2\mathbb L_\varepsilon\phi_\varepsilon-\varepsilon^2N_\varepsilon(\phi_\varepsilon),
	\end{align*}
	where
	\begin{equation*}
		\mathbb L_\varepsilon\phi_\varepsilon:=(-\Delta_{\mathbb S^2})\phi_\varepsilon-\frac{1}{\ep^2}\sum\limits_{m=1}^j \left(V^+_{m,\varepsilon}(\boldsymbol z)-\frac{\kappa_m^+}{2\pi}\ln\frac{1}{\varepsilon}\right)_+^{\gamma-1}\phi_\varepsilon-\frac{1}{\ep^2}\sum\limits_{n=1}^k\left(V^-_{n,\varepsilon}(\boldsymbol z)-\frac{\kappa_n^-}{2\pi}\ln\frac{1}{\varepsilon}\right)_+^{\gamma-1}\phi_\varepsilon
	\end{equation*}
	is the linear term, and
	\begin{align*}
		N_\varepsilon(\phi_\varepsilon)=&\frac{1}{\ep^2}\sum\limits_{m=1}^j\left(\boldsymbol1_{B_\delta(\boldsymbol z_m^+)}\left(\psi_\varepsilon+W_\varepsilon\cos\theta-\mu_{m,\varepsilon}^+\right)_+^\gamma-\left(V^+_{m,\varepsilon}(\boldsymbol z)-\frac{\kappa_m^+}{2\pi}\ln\frac{1}{\varepsilon}\right)_+^\gamma\right.\\
		&\left. \quad-\gamma\left(V^+_{m,\varepsilon}(\boldsymbol z)-\frac{\kappa_m^+}{2\pi}\ln\frac{1}{\varepsilon}\right)_+^{\gamma-1}\phi_\varepsilon\right)\\
		&-\frac{1}{\ep^2}\sum\limits_{n=1}^k\left(\boldsymbol1_{ B_\delta(\boldsymbol z_n^-)}\left(-\psi_\varepsilon-W_\varepsilon\cos\theta-\mu_{n,\varepsilon}^-\right)_+^\gamma-\left(V^-_{n,\varepsilon}(\boldsymbol z)-\frac{\kappa_n^-}{2\pi}\ln\frac{1}{\varepsilon}\right)_+^\gamma\right.\\
		&\left.\quad+\gamma\left(V^-_{n,\varepsilon}(\boldsymbol z)-\frac{\kappa_n^-}{2\pi}\ln\frac{1}{\varepsilon}\right)_+^{\gamma-1}\phi_\varepsilon\right)\\
		:=&N_\varepsilon^+(\phi_\varepsilon)-N_\varepsilon^-(\phi_\varepsilon)
	\end{align*}
	is the nonlinear perturbation. To make $N_\varepsilon(\phi_\varepsilon)$ sufficiently small, the flux constant $\mu_{m,\varepsilon}^+$ and $\mu_{n,\varepsilon}^-$ are chosen to satisfy
	\begin{align*}
		-\frac{\kappa_m^+}{2\pi}\ln\frac{1}{\varepsilon}=&-\sum\limits_{i\neq m}^j\kappa_i^+G(\boldsymbol z_{m,\varepsilon}^+,\boldsymbol z_{i,\varepsilon}^+)-\sum\limits_{l=1}^k\kappa_l^-G(\boldsymbol z_{m,\varepsilon}^+,\boldsymbol z_{l,\varepsilon}^-)\\
		&-\kappa_m^+ H(\boldsymbol z_{m,\varepsilon}^+,\boldsymbol z_{m,\varepsilon}^+)-W\cos\theta_{m,\varepsilon}^++\mu_{m,\varepsilon}^+,
	\end{align*}
	and
	\begin{align*}
		-\frac{\kappa_n^-}{2\pi}\ln\frac{1}{\varepsilon}=&-\sum\limits_{l\neq n}^j\kappa_l^-G(\boldsymbol z_{n,\varepsilon}^-,\boldsymbol z_{l,\varepsilon}^-)-\sum\limits_{i=1}^k\kappa_i^+G(\boldsymbol z_{n,\varepsilon}^-,\boldsymbol z_{i,\varepsilon}^+)\\
		&- \kappa_n^-H(\boldsymbol z_{n,\varepsilon}^-,\boldsymbol z_{n,\varepsilon}^-)-W\cos\theta_{n,\varepsilon}^-+\mu_{n,\varepsilon}^-.
	\end{align*}
	Having done all these preparations, we are now to solve the semilinear problem
	\begin{equation}\label{2-2}
		\mathbb L_\varepsilon\phi_\varepsilon=N_\varepsilon(\phi_\varepsilon).
	\end{equation}
	The common practice to deal with \eqref{2-2} is letting $\mathbb L_\ep^{-1}$ act on both side and apply the fixed point theory on $\phi_\varepsilon=\mathbb L_\ep^{-1}N_\varepsilon(\phi_\varepsilon)$ by estimating the nonlinear perturbation $N_\varepsilon(\phi_\varepsilon)$. However, this approach requires that $\mathbb L_\ep$ is invertible, which is generally not true for our situation. Thus we first prove the existence of $\phi_\ep$ in a projection space, and then solve a finite-dimensional problem in the kernel of $\mathbb L_\ep$, which is known as the Lyapunov--Schmidt reduction.

	\subsection{The linear problem}
	
	To explicit the projection we make for \eqref{2-2}, let us first consider the linearized operator for the generalized plasma problem $-\Delta w_\gamma-(w_\gamma)_+^\gamma=0$ as
	\begin{equation}\label{2-3}
		\mathbb L_0\phi:=-\Delta\phi-(w_\gamma)_+^{\gamma-1}\phi.
	\end{equation}
	Dancer and Yan~\cite{DY} prove the following result 
	\begin{theorem}\label{thm3}
		Let $v\in L^\infty(\mathbb R^2)\cap C(\mathbb R^2)$ be a solution to $\mathbb L _0v=0$ with the operator $\mathbb L_0$ defined in \eqref{2-3}. Then
		\begin{equation*}
			v\in\mathrm{span} \left\{\frac{\partial w_\gamma}{\partial y_1},\frac{\partial w_\gamma}{\partial y_2}\right\}.
		\end{equation*}
	\end{theorem}
	The above theorem gives a precise $2$-dimensional description for the kernel of $\mathbb L_0$. Since $\mathbb L_0$ can be regarded as the limit operator of $\mathbb L_\ep$ after scaling at $\boldsymbol z^\pm_{l,\ep}$, one can hope that the kernel of $\mathbb L_\ep$ can be approximated by a combination of rescaled $\partial_{y_1}w_\gamma$ and $\partial_{y_2}w_\gamma$. To verify this conjecture, we define the tangent mapping $A_l^\pm: (\theta,\varphi)\to (x_1,x_2)$ from $\mathbb S^2$ to $T_{\boldsymbol z_l^\pm}\mathbb S^2$ with the matrix 
	\begin{equation*}
		\mathrm{Mat}(A_l^\pm)=\left(
		\begin{array}{ccc}
			1 & 0             \\
			0 & \sin\theta_{l,\varepsilon}^\pm
		\end{array}
		\right).
	\end{equation*}
	Here, different from the tangent mapping $A$ defined before for $V^\pm_{l,\varepsilon}$, we fix $\sin\theta_{l,\varepsilon}^\pm$ with $\theta_{l,\varepsilon}^\pm$ the first coordinate of $\boldsymbol z^\pm_{l,\ep}$ to simplify the calculation for derivatives blow. Using this notation, we can define 
	\begin{equation*}
		U^\pm_{l,\varepsilon}(\boldsymbol z)=\left\{
		\begin{array}{lll}
			\frac{\kappa_l^\pm}{2\pi}\ln\frac{1}{\varepsilon}+\left(\frac{s_{l,\varepsilon}^\pm}{\ep}\right)^{\frac{\gamma-1}{2}}w_\gamma\left(\frac{A_l^\pm(\boldsymbol z-\boldsymbol z_{l,\varepsilon}^\pm)}{s_{l,\varepsilon}^\pm}\right), \ \ \ &\mathrm{if} \ |A_l^{\pm}(\boldsymbol z-\boldsymbol z^\pm_{l,\varepsilon})|\le s_{l,\varepsilon}^\pm,\\
			\frac{\kappa_l^\pm}{2\pi}\frac{|\ln\varepsilon|}{|\ln s_{l,\varepsilon}^\pm|}\ln|A_l^{\pm}(\boldsymbol z-\boldsymbol z^\pm_{l,\varepsilon})|,&\mathrm{if} \ |A_l^{\pm}(\boldsymbol z-\boldsymbol z^\pm_{l,\varepsilon})|\ge s_{l,\varepsilon}^\pm
		\end{array}
		\right.
	\end{equation*}
	near $V^\pm_{l,\varepsilon}$, and assume that the approximate kernel for $\mathbb L_\varepsilon$ is $(2j+2k)$-dimensional given by the linear combination of
	\begin{equation*}
		X_{m,\varepsilon}^+({\boldsymbol z})=\chi_m^+({\boldsymbol z})\frac{\partial U^+_{m,\varepsilon}}{\partial \theta}, 
		\quad Y_{m,\varepsilon}^+({\boldsymbol z})=\chi_m^+({\boldsymbol z})\frac{\partial U^+_{m,\varepsilon}}{\partial \varphi}, \quad 1\le m\le j,
	\end{equation*}
	\begin{equation*}
		X_{n,\varepsilon}^-({\boldsymbol z})=\chi_n^-({\boldsymbol z})\frac{\partial U^+_{n,\varepsilon}}{\partial \theta}, 
		\quad Y_{n,\varepsilon}^-({\boldsymbol z})=\chi_n^-({\boldsymbol z})\frac{\partial U^+_{n,\varepsilon}}{\partial \varphi}, \quad 1\le n\le k,
	\end{equation*}
	where
	\begin{equation*}
		\chi_l^\pm(\boldsymbol z)=\left\{
		\begin{array}{lll}
			1, \ \ \  & \mathrm{if} \ |\boldsymbol z-\boldsymbol z_{l,\varepsilon}^\pm|_{\mathbb S^2}< \varepsilon|\ln\varepsilon|,\\
			0, & \mathrm{if} \ |\boldsymbol z-\boldsymbol z_{l,\varepsilon}^\pm|_{\mathbb S^2}\ge 2\varepsilon|\ln\varepsilon|
		\end{array}
		\right.
	\end{equation*}
	are smooth truncation functions radially symmetric with respect to $\boldsymbol z_{l,\ep}^\pm$. Moreover, we assume that $\chi_i^\pm(\boldsymbol z)$ satisfy the derivative condition,
	\begin{equation*}
		|\nabla \chi_i^\pm(\boldsymbol z)|\le \frac{2}{\varepsilon|\ln\varepsilon|},\quad \mathrm{and} \quad |\nabla^2 \chi_i^\pm(\boldsymbol z)|\le \frac{2}{\varepsilon^2|\ln\varepsilon|^2}.
	\end{equation*}
	
	By projection with respect to $X_{m,\varepsilon}^+, Y_{m,\varepsilon}^+, X_{n,\varepsilon}^-, Y_{n,\varepsilon}^-$, the projection problem for \eqref{2-2} is
	\begin{equation}\label{2-4}
		\begin{cases}
			\mathbb L_\varepsilon\phi=\mathbf h(\boldsymbol z)+(-\Delta_{\mathbb S^2})\sum_{m=1}^j\left[a_m^+X_{m,\varepsilon}^++b_m^+Y_{m,\varepsilon}^+\right]\\
			\quad\quad\quad\quad\quad\ +(-\Delta_{\mathbb S^2})\sum_{n=1}^k\left[a_n^-X_{n,\varepsilon}^-+b_n^-Y_{n,\varepsilon}^-\right], \ \ &\text{in} \ \mathbb S^2,\\
			\int_{\mathbb S^2}  \phi(\boldsymbol z)(-\Delta_{\mathbb S^2})X_{l,\varepsilon}^\pm(\boldsymbol z) d\boldsymbol \sigma=0, \quad
			\int_{\mathbb S^2}  \phi(\boldsymbol z)(-\Delta_{\mathbb S^2})Y_{i,\varepsilon}^\pm(\boldsymbol z) d\boldsymbol \sigma=0,
		\end{cases}
	\end{equation}
	where $\mathbf h(\boldsymbol z)$ corresponding to the nonlinear term $N_\varepsilon(\phi_\varepsilon)$ satisfies 
	$$\mathrm{supp}\, \mathbf h(\boldsymbol z)\subset \left(\cup_{m=1}^jB_{L\varepsilon}(\boldsymbol z_{m,\varepsilon}^+)\right)\cup \left(\cup_{n=1}^kB_{L\varepsilon}(\boldsymbol z_{n,\varepsilon}^-)\right)$$
	with $L$ a large positive constant, and 
	$$\mathbf\Lambda=(a_1^+,\cdots,a_j^+,b_1^+,\cdots,b_j^+, a_1^-,\cdots,a_k^-,b_1^-,\cdots,b_k^-)$$
	is the $(2j+2k)$-dimensional projection vector determined by
	\begin{footnotesize}
	\begin{equation*}
	a_l^\pm\int_{\mathbb S^2} X_{l,\varepsilon}^\pm(-\Delta_{\mathbb S^2})X_{l,\varepsilon}^\pm d\boldsymbol \sigma=\int_{\mathbb S^2}X_{l,\varepsilon}^\pm\big[\mathbb L_\ep\phi-\mathbf h(\boldsymbol z)\big]d\boldsymbol \sigma, \quad b_l^\pm\int_{\mathbb S^2} Y_{l,\varepsilon}^\pm(-\Delta_{\mathbb S^2})Y_{l,\varepsilon}^\pm d\boldsymbol \sigma=\int_{\mathbb S^2}Y_{l,\varepsilon}^\pm\big[\mathbb L_\ep\phi-\mathbf h(\boldsymbol z)\big]d\boldsymbol \sigma.
    \end{equation*}
    \end{footnotesize}
	
	By denoting the $\|\cdot\|_*$ norm for the function $\phi$ on the sphere as  
	$$\|\phi\|_*=\sup_{\mathbb S^2} |\phi|,$$
	we can prove the following coercive estimate for $\mathbb L_\varepsilon$, which indicates that the projection made in \eqref{2-4} is reasonable.
	\begin{lemma}\label{lem2-2}
		Assume that $\mathbf h$ satisfies $\mathrm{supp}\, \mathbf h\subset \left(\cup_{m=1}^jB_{L\varepsilon}(\boldsymbol z_{m,\varepsilon}^+)\right)\cup \left(\cup_{n=1}^kB_{L\varepsilon}(\boldsymbol z_{n,\varepsilon}^-)\right),$ and $$\varepsilon^{\frac{2}{p}}\|\mathbf h\|_{W^{-1,p}\left(\left(\cup_{m=1}^jB_{L\varepsilon}(\boldsymbol z_{m,\varepsilon}^+)\right)\cup \left(\cup_{n=1}^kB_{L\varepsilon}(\boldsymbol z_{n,\varepsilon}^-)\right)\right)}<\infty$$
		for $p\in (2,+\infty]$, then there exists a small $\varepsilon_0>0$ and a positive constant $c_0$ such that for any $\varepsilon\in(0,\varepsilon_0]$ and solution pair $(\phi,\mathbf\Lambda)$ to \eqref{2-4}, it holds
		\begin{align*}
			\|\phi\|_*+\varepsilon^{1-\frac{2}{p}}\|\nabla\phi&\|_{L^p\left(\left(\cup_{m=1}^jB_{L\varepsilon}(\boldsymbol z_{m,\varepsilon}^+)\right)\cup \left(\cup_{n=1}^kB_{L\varepsilon}(\boldsymbol z_{n,\varepsilon}^-)\right)\right)}+|\mathbf\Lambda|\cdot\varepsilon^{-1}\\
			&\le c_0\varepsilon^{1-\frac{2}{p}}\|\mathbf h\|_{W^{-1,p}\left(\left(\cup_{m=1}^jB_{L\varepsilon}(\boldsymbol z_{m,\varepsilon}^+)\right)\cup \left(\cup_{n=1}^kB_{L\varepsilon}(\boldsymbol z_{n,\varepsilon}^-)\right)\right)}.
		\end{align*}
	\end{lemma}
	\begin{proof}
		From the definition of $X_{m,\ep}^+$, we see $\mathrm{supp}\, X_{m,\ep}^+\subset B_\delta(\boldsymbol z_m^+)$. Thus, to derive the estimate for coefficient $a_m^+$, we can integrate by parts in $B_\delta(\boldsymbol z_m^+)$ to obtain
		\begin{equation}\label{2-5}
			a_m^+\int_{B_\delta(\boldsymbol z_m^+)} \left[(\partial_\theta X_{m,\varepsilon}^+)^2+\left(\frac{\partial_\varphi X_{m,\varepsilon}^+}{\sin\theta}\right)^2\right]d\boldsymbol \sigma=\int_{B_\delta(\boldsymbol z_m^+)}X_{m,\varepsilon}^+\mathbb L_\varepsilon\phi d\boldsymbol \sigma-\int_{B_\delta(\boldsymbol z_m^+)} X_{m,\varepsilon}^+\mathbf h d\boldsymbol \sigma.
		\end{equation}
		For the left hand side of \eqref{2-5}, it holds
		\begin{align*}
		&\quad a_m^+\int_{B_\delta(\boldsymbol z_m^+)} \left[(\partial_\theta X_{m,\varepsilon}^+)^2+\left(\frac{\partial_\varphi X_{m,\varepsilon}^+}{\sin\theta}\right)^2\right]d\boldsymbol \sigma\\
			&=2\Lambda\int_{B_\delta(\boldsymbol z_m^+)} \sin\theta \left[\left(\frac{\partial^2 U^+_{1,\varepsilon}}{\partial \theta^2}\right)^2+\left(\frac{\frac{\partial^2U_{1,\varepsilon}^+}{\partial_\varphi\partial_\theta}}{\sin\theta}\right)^2\right]d\theta d\varphi+\frac{o_\varepsilon(1)}{\varepsilon^2}\\
			&=\Lambda\frac{C_z}{\varepsilon^2} (1+o_\varepsilon(1)),
		\end{align*}
		where $C_z=\int_{\mathbb R^2}(\nabla \partial_{y_1}w)^2 d\boldsymbol y$ is a positive constant.

		To estimate the first term in the right-hand side of \eqref{2-5}, we have
		\begin{align*}
			&\quad\int_{B_\delta(\boldsymbol z_m^+)}X_{m,\varepsilon}^+\mathbb L_\varepsilon\phi d\boldsymbol \sigma=\int_{\Pi^+}\phi\mathbb L_\varepsilon X_{m,\varepsilon}^+ d\boldsymbol \sigma\\
			&=\int_{B_\delta(\boldsymbol z_m^+)}\phi\mathbb (-\Delta_{\mathbb S^2})X_{m,\varepsilon}^+ d\boldsymbol \sigma-\frac{\gamma}{\ep^2}\int_{B_\delta(\boldsymbol z_m^+)}\phi \left(V^+_{m,\varepsilon}(\boldsymbol z)-\frac{\kappa_m^+}{2\pi}\ln\frac{1}{\varepsilon}\right)_+^{\gamma-1}\frac{\partial U^+_{m,\varepsilon}}{\partial \theta}d\boldsymbol \sigma.
		\end{align*}
		It holds
		\begin{align*}
			&\quad\int_{B_\delta(\boldsymbol z_m^+)}\phi\mathbb (-\Delta_{\mathbb S^2})X_{m,\varepsilon}^+ d\boldsymbol \sigma=\int_{B_\delta(\boldsymbol z_m^+)}\phi\left[\partial_\theta(\sin\theta\partial_\theta X_{m,\varepsilon}^+)+\frac{\partial_\varphi^2X_{m,\varepsilon}^+}{\sin\theta}\right]d\theta d\varphi\\
			&=\int_{B_\delta(\boldsymbol z_m^+)}\phi\cos\theta\partial_\theta\left(\chi_m^+\frac{\partial U^+_{m,\varepsilon}}{\partial \theta}\right)d\theta d\varphi +\int_{B_\delta(\boldsymbol z_m^+)}\phi\sin\theta\left[(\partial_\theta^2\chi_m^+)\frac{\partial U^+_{m,\varepsilon}}{\partial \theta}+\frac{\partial_\varphi^2\chi_m^+}{\sin^2\theta}\frac{\partial U^+_{m,\varepsilon}}{\partial \theta}\right]d\theta d\varphi\\
			&\quad+\int_{B_\delta(\boldsymbol z_m^+)}\phi\sin\theta\left[\partial_\theta\chi_m^+\frac{\partial^2 U^+_{m,\varepsilon}}{\partial \theta^2}+\frac{\partial_\varphi\chi_m^+}{\sin^2\theta} \frac{\partial^2 U^+_{m,\varepsilon}}{\partial\varphi\partial \theta}\right]d\theta d\varphi\\
			&\quad+\int_{B_\delta(\boldsymbol z_m^+)}\phi\chi_m^+\sin\theta\left[\frac{1}{\sin^2\theta}-\frac{1}{\sin^2\theta_0}\right] \frac{\partial^2 U^+_{m,\varepsilon}}{\partial\varphi^2\partial \theta}d\theta d\varphi+\int_{B_\delta(\boldsymbol z_m^+)}\phi\chi_1^+\sin\theta\left[\frac{\partial^3 U^+_{m,\varepsilon}}{\partial \theta^3}+\frac{\frac{\partial^3 U^+_{m,\varepsilon}}{\partial \varphi^2\partial \theta}}{\sin^2\theta_0}\right]d\theta d\varphi\\
			&=I_1+I_2+I_3+I_4+I_5.
		\end{align*}
		According to the property for truncation function $\chi_1^+$, we have $I_1\le C|\ln\varepsilon|\|\phi\|_*$, $I_2\le \frac{C\|\phi\|_*}{\varepsilon|\ln\varepsilon|}$, $I_3\le \frac{C\|\phi\|_*}{\varepsilon|\ln\varepsilon|}$, and $I_4\le C\|\phi\|_*$. For the last term $I_5$, by applying the non-degenerate property in Theorem \ref{thm3}, we have
		\begin{align*}
			I_5&=\frac{\gamma}{\ep^2}\int_{B_\delta(\boldsymbol z_m^+)}\phi \left(U^+_{m,\varepsilon}(\boldsymbol z)-\frac{\kappa_m^+}{2\pi}\ln\frac{1}{\varepsilon}\right)_+^{\gamma-1}\frac{\partial U^+_{m,\varepsilon}}{\partial \theta}d\boldsymbol \sigma\\
			&=\frac{\gamma}{\ep^2}\int_{B_\delta(\boldsymbol z_m^+)}\phi \left(V^+_{m,\varepsilon}(\boldsymbol z)-\frac{\kappa_m^+}{2\pi}\ln\frac{1}{\varepsilon}\right)_+^{\gamma-1}\frac{\partial U^+_{m,\varepsilon}}{\partial \theta}d\boldsymbol \sigma+O_\varepsilon(1)\|\phi\|_*.
		\end{align*} 
		Thus we have
		\begin{equation*}
			\int_{B_\delta(\boldsymbol z_m^+)}X_{m,\varepsilon}^+\mathbb L_\varepsilon\phi d\boldsymbol \sigma\le \frac{C}{\varepsilon|\ln\varepsilon|}\|\phi\|_*.
		\end{equation*}
		For the last term in \eqref{2-5}, we can apply the Poincar\'e inequality to derive 
		\begin{align*}
			\int_{B_\delta(\boldsymbol z_m^+)}X_{m,\varepsilon}^+ \mathbf h d\boldsymbol \sigma&\le \|\mathbf h\|_{W^{-1,p}\left(B_{L\varepsilon}(\boldsymbol z_{m,\varepsilon}^+)\right)}\|\nabla X_{m,\varepsilon}^+\|_{L^{p'}\left(B_{L\varepsilon}(\boldsymbol z_{m,\varepsilon}^+)\right)}\\
			&\le C\varepsilon^{\frac{2}{p'}-2} \|\mathbf h\|_{W^{-1,p}\left(B_{L\varepsilon}(\boldsymbol z_{m,\varepsilon}^+)\right)}.
		\end{align*}
		Concluding all the above estimates together, we have
		\begin{equation*}
			|a_m^+|\varepsilon^{-1}\le\frac{C}{|\ln\varepsilon|}\|\phi\|_*+\varepsilon^{\frac{2}{p'}-1}\|\mathbf h\|_{W^{-1,p}\left(B_{L\varepsilon}(\boldsymbol z_{m,\varepsilon}^+)\right)}.
		\end{equation*}
		Similarly, we can deal with $Y_{m,\ep}^+$, $X_{n,\ep}^-$ and $Y_{n,\ep}^-$ to obtain
		\begin{equation*}
			|\mathbf \Lambda|\varepsilon^{-1}\le\frac{C}{|\ln\varepsilon|}\|\phi\|_*+\varepsilon^{\frac{2}{p'}-1}\|\mathbf h\|_{W^{-1,p}\left(\left(\cup_{m=1}^jB_{L\varepsilon}(\boldsymbol z_{m,\varepsilon}^+)\right)\cup \left(\cup_{n=1}^kB_{L\varepsilon}(\boldsymbol z_{n,\varepsilon}^-)\right)\right)}.
		\end{equation*}
		Since 
		\begin{equation*}
			\|(-\Delta_{\mathbb S^2}) X_{l,\varepsilon}^\pm\|_{W^{-1,p}(B_{L\varepsilon}(\boldsymbol z_l^\pm)}\le C\|\nabla X_{m,\varepsilon}^+\|_{L^{p}(B_{L\varepsilon}(\boldsymbol z_j^\pm)}=C \varepsilon^{\frac{2}{p}-2},
		\end{equation*} 
		and 
		\begin{equation*}
			\|(-\Delta_{\mathbb S^2})Y_{l,\varepsilon}^\pm\|_{W^{-1,p}(B_{L\varepsilon}(\boldsymbol z_1^+)}\le C\|\nabla Y_{l,\varepsilon}^\pm\|_{L^{p}(B_{L\varepsilon}(\boldsymbol z_l^\pm)}=C \varepsilon^{\frac{2}{p}-2},
		\end{equation*} 
		we have
		\begin{align*}
			&\quad\left\|\sum_{m=1}^j\left[a_m^+(-\Delta_{\mathbb S^2}) X_{m,\varepsilon}^++b_m^+(-\Delta_{\mathbb S^2})Y_{m,\varepsilon}^+\right]\right\|_{W^{-1,p}\left(\left(\cup_{m=1}^jB_{L\varepsilon}(\boldsymbol z_{m,\varepsilon}^+)\right)\right)}\\
			&\quad+\left\|\sum_{n=1}^k\left[a_n^-(-\Delta_{\mathbb S^2}) X_{n,\varepsilon}^-+b_n^-(-\Delta_{\mathbb S^2})Y_{n,\varepsilon}^-\right]\right\|_{W^{-1,p}\left(\left(\cup_{n=1}^kB_{L\varepsilon}(\boldsymbol z_{n,\varepsilon}^-)\right)\right)}\\
			&\le C \frac{\varepsilon^{\frac{2}{p}-1}}{|\ln\varepsilon|}\|\phi\|_*+C\|\mathbf h\|_{W^{-1,p}\left(\left(\cup_{m=1}^jB_{L\varepsilon}(\boldsymbol z_{m,\varepsilon}^+)\right)\cup \left(\cup_{n=1}^kB_{L\varepsilon}(\boldsymbol z_{n,\varepsilon}^-)\right)\right)}.
		\end{align*}
		
		To obtain the estimate for $\|\phi\|_*$ and $\varepsilon^{1-\frac{2}{p}}\|\nabla\phi\|_{L^p\left(\left(\cup_{m=1}^jB_{L\varepsilon}(\boldsymbol z_{m,\varepsilon}^+)\right)\cup \left(\cup_{n=1}^kB_{L\varepsilon}(\boldsymbol z_{n,\varepsilon}^-)\right)\right)}$ in the next step, we argue by contradiction. Suppose not. Then there exists a sequence $\{\varepsilon_{\mathbf n}\}$ tending to $0$, such that
		\begin{equation}\label{2-6}
			\|\phi_{\mathbf n}\|_*+\varepsilon_{\mathbf n}^{1-\frac{2}{p}}\|\nabla\phi_{\mathbf n}\|_{L^p\left(\left(\cup_{m=1}^jB_{L\varepsilon_{\mathbf n}}(\boldsymbol z_{m,\varepsilon_{\mathbf n}}^+)\right)\cup \left(\cup_{n=1}^kB_{L\varepsilon_{\mathbf n}}(\boldsymbol z_{n,\varepsilon_{\mathbf n}}^-)\right)\right)}=1,
		\end{equation} 
		and
		\begin{equation*}
			\varepsilon_{\mathbf n}^{\frac{2}{p'}-2} \|\mathbf h\|_{W^{-1,p}\left(\left(\cup_{m=1}^jB_{L\varepsilon_{\mathbf n}}(\boldsymbol z_{m,\varepsilon_{\mathbf n}}^+)\right)\cup \left(\cup_{n=1}^kB_{L\varepsilon_{\mathbf n}}(\boldsymbol z_{n,\varepsilon_{\mathbf n}}^-)\right)\right)}\le\frac{1}{\mathbf n}.
		\end{equation*}
		The solution $\phi_{\mathbf n}$ satisfies the equation
		\begin{align*}
			&(-\Delta_{\mathbb S^2})\phi_{\mathbf n}=-\frac{\gamma}{\ep^2_{\mathbf n}}\sum\limits_{m=1}^j \left(V^+_{m,\varepsilon_{\mathbf n}}(\boldsymbol z)-\frac{\kappa_m^+}{2\pi}\ln\frac{1}{\varepsilon_{\mathbf n}}\right)_+^{\gamma-1}\phi_{\mathbf n}-\frac{\gamma}{\ep^2_{\mathbf n}}\sum\limits_{n=1}^k\left(V^-_{n,\varepsilon_{\mathbf n}}(\boldsymbol z)-\frac{\kappa_n^-}{2\pi}\ln\frac{1}{\varepsilon_{\mathbf n}}\right)_+^{\gamma-1}\phi_{\mathbf n}\\
			&+\mathbf h+\sum_{m=1}^j\left[a_m^+(-\Delta_{\mathbb S^2}) X_{m,\varepsilon_{\mathbf n}}^++b_m^+(-\Delta_{\mathbb S^2})Y_{m,\varepsilon_{\mathbf n}}^+\right]+\sum_{n=1}^k\left[a_n^-(-\Delta_{\mathbb S^2}) X_{n,\varepsilon_{\mathbf n}}^-+b_n^-(-\Delta_{\mathbb S^2})Y_{n,\varepsilon_{\mathbf n}}^-\right].
		\end{align*}
		By letting
		\begin{align*}
			&g(\boldsymbol z)=-\frac{\gamma}{\ep^2_{\mathbf n}}\sum\limits_{m\neq l}^j \left(V^+_{m,\varepsilon_{\mathbf n}}(\boldsymbol z)-\frac{\kappa_m^+}{2\pi}\ln\frac{1}{\varepsilon_{\mathbf n}}\right)_+^{\gamma-1}\phi_{\mathbf n}-\frac{\gamma}{\ep^2_{\mathbf n}}\sum\limits_{n=1}^k\left(V^-_{n,\varepsilon_{\mathbf n}}(\boldsymbol z)-\frac{\kappa_n^-}{2\pi}\ln\frac{1}{\varepsilon_{\mathbf n}}\right)_+^{\gamma-1}\phi_{\mathbf n}\\
			&+\mathbf h+\sum_{m=1}^j\left[a_m^+(-\Delta_{\mathbb S^2}) X_{m,\varepsilon_{\mathbf n}}^++b_m^+(-\Delta_{\mathbb S^2})Y_{m,\varepsilon_{\mathbf n}}^+\right]+\sum_{n=1}^k\left[a_n^-(-\Delta_{\mathbb S^2}) X_{n,\varepsilon_{\mathbf n}}^-+b_n^-(-\Delta_{\mathbb S^2})Y_{n,\varepsilon_{\mathbf n}}^-\right].
		\end{align*}
		and 
		$$\tilde v(\boldsymbol y)=v(s_{\varepsilon_{\mathbf n}}^+y_1+\theta_{l,\ep_{\mathbf n}}^+,s_{\varepsilon_{\mathbf n}^+}\sin^{-1}(s_{\varepsilon_{\mathbf n}}^+y_1+\theta_0)y_2+\varphi_{l,\ep_{\mathbf n}}^+)$$
		for a general function $v$, we see that $\tilde\phi_n$ satisfies the scaled equation near $\boldsymbol z_{l,\ep_{\mathbf n}}^+$
		\begin{equation}\label{2-7}
			\begin{split}
			&\quad\int_{D_{\mathbf n}}\left[\frac{\partial \tilde\phi_{\mathbf n}}{\partial y_1}\frac{\partial \zeta}{\partial y_1}+\frac{\partial \tilde\phi_{\mathbf n}}{\partial y_2}\frac{\partial \zeta}{\partial y_2}\right]d\boldsymbol y\\
			&=\gamma\int_{B_1(\boldsymbol 0)}(w_\gamma)_+^{\gamma-1}\tilde\phi_{\mathbf n}\zeta d\xi+\langle \tilde g_{\mathbf n}, \zeta\rangle, \quad \forall \, \zeta\in C_0^\infty(D_{\mathbf n}),
			\end{split}
		\end{equation}
		where
		$$D_n:=\{\boldsymbol y\mid (s_{\varepsilon_{\mathbf n}}^+y_1+\theta_{l,\ep_{\mathbf n}}^+,s_{\varepsilon_{\mathbf n}^+}\sin^{-1}(s_{\varepsilon_{\mathbf n}}^+y_1+\theta_0)y_2+\varphi_{l,\ep_{\mathbf n}}^+)\in B_\delta(\boldsymbol z_{l,\ep_{\mathbf n}}^+)\}.$$
		Since 
		\begin{equation*}
			\|\tilde g_{\mathbf n}\|_{W^{-1,p}(B_L(\boldsymbol 0))}\le C\varepsilon_{\mathbf n}^{1-\frac{2}{p}}\left( \frac{\varepsilon_{\mathbf n}^{\frac{2}{p}-1}}{|\ln\varepsilon_{\mathbf n}|}\|\phi\|_*+\|\mathbf h\|_{W^{-1,p}\left(\left(\cup_{m=1}^jB_{L\varepsilon_{\mathbf n}}(\boldsymbol z_{m,\varepsilon_{\mathbf n}}^+)\right)\cup \left(\cup_{n=1}^kB_{L\varepsilon_{\mathbf n}}(\boldsymbol z_{n,\varepsilon_{\mathbf n}}^-)\right)\right)}\right)=o_{\mathbf n}(1)
		\end{equation*}
		by estimate \eqref{2-7} and the boundedness of $\gamma(w_\gamma)_+^{\gamma-1}$, $\tilde\phi_{\mathbf n}$ is bounded in $W^{1,p}_{\mathrm{loc}}(\mathbb R^2)$ by regularity theory of elliptic operators, and bounded in $C^{\alpha}_{\mathrm{loc}}(\mathbb R^2)$ for some $\alpha>0$ by Sobolev embedding. Hence we can assume $\tilde\phi_{\mathbf n}$ converge uniformly to $\phi^*\in L^\infty(\mathbb R^2)\cap C(\mathbb R^2)$ in any fixed compact set in $\mathbb R^2$, which satisfies the limit equation
		\begin{equation*}
			-\Delta\phi^*=\gamma(w_\gamma)_+^{\gamma-1}\phi^*, \quad \mathrm{in} \ \mathbb R^2.
		\end{equation*}
		Owing to Theorem \ref{thm3}, it holds
		\begin{equation*}
			\phi^*=C_1\frac{\partial w}{\partial y_1}+C_2\frac{\partial w}{\partial y_2}.
		\end{equation*}
		However, from the orthogonal condition in \eqref{2-4} we know that $\int_{\mathbb R^2}\nabla\phi^*\nabla \partial_{y_1} wd\boldsymbol y=0$ and $\int_{\mathbb R^2}\nabla\phi^*\nabla \partial_{y_2} wd\boldsymbol y=0$. Thus $C_1=C_2=0$ and $\|\phi_{\mathbf n}\|_{L^\infty(B_{L_{\mathbf n}}(\boldsymbol z_{l,\ep_{\mathbf n}}^+))}=o_{\mathbf n}(1)$. Then we can deal with other vortices by the same method to obtain
		\begin{equation*}
			\|\phi_{\mathbf n}\|_{L^\infty(B_{L_{\mathbf n}}\left(\left(\cup_{m=1}^jB_{L\varepsilon_{\mathbf n}}(\boldsymbol z_{m,\varepsilon_{\mathbf n}}^+)\right)\cup \left(\cup_{n=1}^kB_{L\varepsilon_{\mathbf n}}(\boldsymbol z_{n,\varepsilon_{\mathbf n}}^-)\right)\right)}=o_{\mathbf n}(1).
		\end{equation*}
		Then, because of the maximum principle, we have
		\begin{equation}\label{2-8}
			\|\phi_{\mathbf n}\|_*=o_{\mathbf n}(1).
		\end{equation}
		
		On the other hand, the right-hand side of equation \eqref{2-7} can be controlled by 
		$$o_{\mathbf n}(1)\|\zeta\|_{W^{1,1}(B_L(\boldsymbol 0))}+o_{\mathbf n}(1)\|\zeta\|_{W^{1,p'}(B_L(\boldsymbol 0))}=o_{\mathbf n}(1)\left(\int_{B_L(\boldsymbol 0)}|\nabla\zeta|^{p'}\right)^{\frac{1}{p'}}d\boldsymbol y.$$
		Consequently, we deduce that $\varepsilon_{\mathbf n}^{1-\frac{2}{p}}\|\nabla\phi_{\mathbf n}\|_{L^p(B_{L_{\mathbf n}}(\boldsymbol z_{l,\ep_{\mathbf n}}^+))}\le C||\nabla\tilde\phi_{\mathbf n}||_{L^p(B_L(\boldsymbol 0))}=o_{\mathbf n}(1),$ and
		\begin{equation*}
			\varepsilon_{\mathbf n}^{1-\frac{2}{p}}\|\nabla\phi_{\mathbf n}\|_{L^p(B_{L_{\mathbf n}}\left(\left(\cup_{m=1}^jB_{L\varepsilon_{\mathbf n}}(\boldsymbol z_{m,\varepsilon_{\mathbf n}}^+)\right)\cup \left(\cup_{n=1}^kB_{L\varepsilon_{\mathbf n}}(\boldsymbol z_{n,\varepsilon_{\mathbf n}}^-)\right)\right)}=o_{\mathbf n}(1),
		\end{equation*}
		which result in a contradiction to \eqref{2-6} together with \eqref{2-8}. Combining this fact with the estimate for $\mathbf \Lambda$, we then complete the proof of Lemma \ref{lem2-2}.
	\end{proof}
	
	\bigskip
	
	By the coercive estimate in Lemma \ref{lem2-2}, the projection problem \eqref{2-4} is solvable according to the following lemma.
	\begin{lemma}\label{lem2-3}
		Suppose that $\mathrm{supp}\, \mathbf h\subset \left(\cup_{m=1}^jB_{L\varepsilon}(\boldsymbol z_{m,\varepsilon}^+)\right)\cup \left(\cup_{n=1}^kB_{L\varepsilon}(\boldsymbol z_{n,\varepsilon}^-)\right)$ and 
		$$\varepsilon^{1-\frac{2}{p}}\|\mathbf h\|_{W^{-1,p}\left(\left(\cup_{m=1}^jB_{L\varepsilon}(\boldsymbol z_{m,\varepsilon}^+)\right)\cup \left(\cup_{n=1}^kB_{L\varepsilon}(\boldsymbol z_{n,\varepsilon}^-)\right)\right)}<\infty$$
		for $p\in(2,+\infty]$. Then there exists a small $\varepsilon_0>0$ such that for any $\varepsilon\in(0,\varepsilon_0]$ and $\mathbf\Lambda$ the projection vector, \eqref{2-4} has a unique solution $\phi_\varepsilon=\mathcal T_\varepsilon \, \mathbf h$, where $\mathcal T_\varepsilon$ is a linear operator. Moreover, there exists a constant $c_0>0$ independent of $\varepsilon$, such that
		\begin{equation}\label{2-9}
			\begin{split}
			\|\phi_\varepsilon\|_*+\varepsilon^{1-\frac{2}{p}}&\|\nabla\phi_\varepsilon\|_{L^p(\left(\cup_{m=1}^jB_{L\varepsilon}(\boldsymbol z_{m,\varepsilon}^+)\right)\cup \left(\cup_{n=1}^kB_{L\varepsilon}(\boldsymbol z_{n,\varepsilon}^-)\right))}\\
			&\le c_0\varepsilon^{1-\frac{2}{p}}\|\mathbf h\|_{W^{-1,p}(\left(\cup_{m=1}^jB_{L\varepsilon}(\boldsymbol z_{m,\varepsilon}^+)\right)\cup \left(\cup_{n=1}^kB_{L\varepsilon}(\boldsymbol z_{n,\varepsilon}^-)\right))}.
			\end{split}
		\end{equation}
	\end{lemma}
	\begin{proof}
		Let $H(\mathbb S^2)$ be the Hilbert space endowed with the inner product
		\begin{equation*}
			[u,v]_{H(\mathbb{S}^2)}=\int_{\mathbb S^2}\left[\frac{\partial u}{\partial\theta}\frac{\partial v}{\partial \theta}+\frac{1}{\sin^2\theta}\frac{\partial u}{\partial\varphi}\frac{\partial v}{\partial \varphi}\right]d\boldsymbol\sigma.
		\end{equation*}
		We also introduce two spaces as follows. The first one is
		\begin{equation*}
			E_\varepsilon=\left\{u\in H(\mathbb S^2)\,\,\, \big|\, \,\, \|u\|_*<\infty, \ [u,X_{l,\varepsilon}^\pm]_{H(\mathbb S^2)}=0, \ [u,Y_{l,\varepsilon}^\pm]_{H(\mathbb S^2)}=0\right\}
		\end{equation*}
		with norm $||\cdot||_*$, and the second one is
		\begin{equation*}
			F_\varepsilon=\left\{\mathbf h^* \in W^{-1,p}(\left(\cup_{m=1}^jB_{L\varepsilon}(\boldsymbol z_{m,\varepsilon}^+)\right)\cup \left(\cup_{n=1}^kB_{L\varepsilon}(\boldsymbol z_{n,\varepsilon}^-)\right)) \, \big| \, \int_{\mathbb S^2}X_{l,\varepsilon}^\pm\mathbf h^*d\boldsymbol \sigma=0, \ \int_{\mathbb S^2}Y_{l,\varepsilon}^\pm\mathbf h^*d\boldsymbol \sigma=0\right\}
		\end{equation*}
		with $p\in (2,+\infty]$. Then for $\phi_\varepsilon\in E_\varepsilon$, the problem \eqref{2-6} has an equivalent operator form.
		\begin{equation*}
			\begin{split}
				\phi_\varepsilon&=(-\Delta_{\mathbb S^2})^{-1}P_\varepsilon\left(-\frac{1}{\ep^2}\sum\limits_{m=1}^j \left(V^+_{m,\varepsilon}-\frac{\kappa_m^+}{2\pi}\ln\frac{1}{\varepsilon}\right)_+^{\gamma-1}\phi_\varepsilon-\frac{1}{\ep^2}\sum\limits_{n=1}^k\left(V^-_{n,\varepsilon}-\frac{\kappa_n^-}{2\pi}\ln\frac{1}{\varepsilon}\right)_+^{\gamma-1}\phi_\varepsilon\right)\\
				&\quad+(-\Delta_{\mathbb S^2})^{-1} P_\varepsilon \mathbf h\\
				&=\mathscr K\phi_\varepsilon+(-\Delta_{\mathbb S^2})^{-1} P_\varepsilon \mathbf h,
			\end{split}
		\end{equation*}
		where
		\begin{equation*}
			(-\Delta_{\mathbb S^2})^{-1}u:=\int_{\mathbb S^2}G(\boldsymbol z,\boldsymbol z')u(\boldsymbol z')d\boldsymbol \sigma',
		\end{equation*}
		and $P_\varepsilon$ is the projection operator to $F_\varepsilon$. Since $X_{l,\varepsilon}^\pm$ and $Y_{l,\varepsilon}^\pm$ have a compact support owing 
		to the truncation $\chi_l^\pm(\boldsymbol z)$, by the definition of $G(\boldsymbol x,\boldsymbol x')$, 
		we see that $\mathscr K$ maps $E_\varepsilon$ to $E_\varepsilon$.
		
		Notice that $\mathscr K$ is a compact operator. Because of the Fredholm alternative, \eqref{2-4} has a unique solution if the homogeneous equation
		\begin{equation*}
			\phi_\varepsilon=\mathscr K\phi_\varepsilon
		\end{equation*}
		has only trivial solution in $E_\varepsilon$, which can be obtained from Lemma \ref{lem2-2}. Now we can let
		$$\mathcal T_\varepsilon:=(\text{Id}-\mathscr K)^{-1}(-\Delta_{\mathbb S^2})^{-1} P_\varepsilon,$$
		and the estimate \eqref{2-9} holds by Lemma \ref{lem2-2}. The proof is thus complete.
	\end{proof}
	
	\subsection{The reduction}
	
	Now we can apply the contraction mapping theorem to obtain a solution $\phi_\ep$ to \eqref{2-4} with $\mathbf h(\boldsymbol z)=N_\varepsilon(\phi_\varepsilon)$. For this purpose, we study the nonlinear term $N_\varepsilon(\phi_\varepsilon)$, which needs an estimate for the vorticity sets $\Omega_{m,\varepsilon}^+$ and $\Omega_{n,\varepsilon}^-$. 
	
	By denoting  
	$$\tilde v_m^+(\boldsymbol y)=v(s_{m,\varepsilon}^+ y_1+\theta_{m,\varepsilon}^+,s_{m,\varepsilon}^+\sin^{-1}(s_{m,\varepsilon}^+ y_1+\theta_{m,\ep}^+)y_2+\varphi_{m,\ep}^+),$$
	and 
	$$\tilde v_n^-(\boldsymbol y)=v(s_{n,\varepsilon}^- y_1+\theta_{n,\ep}^-,s_{n,\varepsilon}^-\sin^{-1}(s_{n,\varepsilon}^- y_1+\theta_{n,\ep}^-)y_2+\varphi_{n,\ep}^-)$$
	for a general function $v$, we have the following lemma concerning the estimate for level sets 
	$$\{\boldsymbol z\in B_\delta(\boldsymbol z_{m}^+)\mid \psi_\varepsilon(\boldsymbol z)+W\sin\theta_0=\mu_{m,\varepsilon}^+\} \quad \mathrm{and} \quad \{\boldsymbol z\in B_\delta(\boldsymbol z_{n}^-)\mid -\psi_\varepsilon(\boldsymbol z)-W\sin\theta_0=\mu_{n,\varepsilon}^-\}$$
	as the perturbations of $B_{s_{m,\ep}^+}(\boldsymbol z_{m,\ep}^+)$ and $B_{s_{n,\ep}^-}(\boldsymbol z_{m,\ep}^-)$.
	\begin{lemma}\label{lem2-4}
		Suppose that $\phi$ is a function satisfying
		\begin{equation}\label{2-10}
			\|\phi_\varepsilon\|_*+\varepsilon^{1-\frac{2}{p}}\|\nabla\phi_\varepsilon\|_{L^p\left(\left(\cup_{m=1}^jB_{L\varepsilon}(\boldsymbol z_{m,\varepsilon}^+)\right)\cup \left(\cup_{n=1}^kB_{L\varepsilon}(\boldsymbol z_{n,\varepsilon}^-)\right)\right)}= O(\varepsilon)
		\end{equation}
		with $p\in (2,+\infty]$. Then the sets
		$$\tilde{\mathbf\Gamma}_{m,\phi}^+:=\{\boldsymbol y \mid \tilde\Psi_{m,\varepsilon}^+(\boldsymbol y)+\tilde\phi_m^++W\cos(s_{m,\varepsilon }^+y_1+\theta_{m,\varepsilon}^+)=\mu_{m,\varepsilon}^+\}$$
		and
		$$\tilde{\mathbf\Gamma}_{n,\phi}^-:=\{\boldsymbol y \mid -\tilde\Psi_{n,\varepsilon}^-(\boldsymbol y)-\tilde\phi_n^--W\cos(s_{n,\varepsilon }^-y_1+\theta_{n,\varepsilon}^-)=\mu_{n,\varepsilon}^-\}$$
		are $C^1$ closed curves in $\mathbb{R}^2$, and
		\begin{equation}\label{2-11}
			\begin{split}
				\tilde{\mathbf\Gamma}_{m,\phi}^+(\xi)&=(1+t_\varepsilon(\xi))(\cos\xi,\sin\xi)\\
				&=(1+t_{\varepsilon,\phi}(\xi)+t_{\varepsilon,\mathcal F_m^+}(\xi)+O(\varepsilon^2))(\cos\xi,\sin\xi)\\
				&=\left(1+\frac{1}{s_{m,\varepsilon}^+\beta_{m,\varepsilon}^+}\tilde\phi(\cos\xi,\sin\xi)\right)(\cos\xi,\sin\xi)+\frac{\mathcal F^+_m(\xi)}{s_{m,\varepsilon}^+\beta_{m,\varepsilon}^+}(\cos\theta,0)\\
				&\quad+O(\varepsilon^2), \quad \xi\in (0,2\pi],
			\end{split}
		\end{equation}
			\begin{equation}\label{2-12}
			\begin{split}
				\tilde{\mathbf\Gamma}_{n,\phi}^-(\xi)&=(1+t_\varepsilon(\xi))(\cos\xi,\sin\xi)\\
				&=(1+t_{\varepsilon,\phi}(\xi)+t_{\varepsilon,\mathcal F_m^-}(\xi)+O(\varepsilon^2))(\cos\xi,\sin\xi)\\
				&=\left(1+\frac{1}{s_{n,\varepsilon}^-\beta_{n,\varepsilon}^-}\tilde\phi(\cos\xi,\sin\xi)\right)(\cos\xi,\sin\xi)+\frac{\mathcal F^-_n(\xi)}{s_{n,\varepsilon}^-\beta_{n,\varepsilon}^-}(\cos\theta,0)\\
				&\quad+O(\varepsilon^2), \quad \xi\in (0,2\pi]
			\end{split}
		\end{equation}
		with $\|t_\varepsilon\|_{C^1[0,2\pi)}=O(\ep)$, and 
		\begin{align*}
			\mathcal F_m^+(\xi)&=\bigg\langle \sum\limits_{l\neq m}^j\kappa_{l}^+\nabla G(\boldsymbol z_{m,\ep}^+,\boldsymbol z^+_{l,\ep})+\kappa_m^+ \nabla H(\boldsymbol z_{m,\ep}^+,\boldsymbol z^+_{m,\ep})-\sum\limits_{n=1}^k\kappa_n^-\nabla G(\boldsymbol z_{m,\ep}^+,\boldsymbol z^-_{l,\ep})\\
			&\quad-(W\sin\theta_{m,\ep}^+,0),\quad (s_{m,\varepsilon}^+\cos\xi,s_{m,\varepsilon}^+\sin^{-1}(s_{m,\varepsilon}^+ y_1+\theta_{m,\ep}^+)\sin\xi)\bigg\rangle,
		\end{align*}
		\begin{align*}
			\mathcal F_n^-(\xi)&=\bigg\langle \sum\limits_{l\neq n}^k\kappa_{l}^-\nabla G(\boldsymbol z_{n,\ep}^-,\boldsymbol z^-_{l,\ep})+\kappa_n^- \nabla H(\boldsymbol z_{n,\ep}^-,\boldsymbol z^-_{n,\ep})-\sum\limits_{m=1}^j\kappa_m^+\nabla G(\boldsymbol z_{n,\ep}^-,\boldsymbol z^+_{m,\ep})\\
			&\quad-(W\sin\theta_{m,\ep}^+,0),\quad (s_{n,\varepsilon}^-\cos\xi,s_{n,\varepsilon}^-\sin^{-1}(s_{n,\varepsilon}^- y_1+\theta_{n,\ep}^-)\sin\xi)\bigg\rangle.
		\end{align*}
		Moreover, it holds
		\begin{equation}\label{}
			\left|\tilde{\mathbf\Gamma}_{l,\phi_1}^\pm-\tilde{\mathbf\Gamma}_{l,\phi_2}^\pm\right|=\left(\frac{1}{s_{l,\varepsilon}^\pm\beta_{l,\varepsilon}^\pm}+O(\varepsilon)\right)\|\tilde\phi_{l,1}^\pm-\tilde\phi_{l,2}^\pm\|_{L^\infty(B_L(\boldsymbol 0))}
		\end{equation}
		for two functions $\phi_1,\phi_2$ satisfying \eqref{2-10}.
	\end{lemma}
	\begin{proof}
		We prove the expansion for $\tilde{\mathbf\Gamma}_{m,\phi}^+$, while that for $\tilde{\mathbf\Gamma}_{n,\phi}^-$ is similar. Let $(y_1,y_2)=(\cos\xi,\sin\xi)$, and the inner product $\mathcal F_m^+(\xi)$ is a regular term of order $O(\varepsilon)$. Then by our choice of $\mu_{l,\varepsilon}^+$ before  \eqref{2-2}, we have
		\begin{equation*}
			\tilde\psi_{m,\varepsilon}^+(\boldsymbol y)-W\cos(s_{m,\varepsilon}^+ y_1+\theta_{m,\ep}^+)-\mu_{m,\varepsilon}^+=\tilde\phi_m^++\mathcal F_m^++O(\varepsilon^2).
		\end{equation*} 
		According to the gradient condition \eqref{2-1}, it holds
		\begin{align*}
			&\nabla\left[\tilde\psi_{m,\varepsilon}^+(\boldsymbol y)-W\cos(s_{m,\varepsilon}^+ y_1+\theta_{m,\ep}^+)-\mu_{m,\varepsilon}^+\right]\bigg|_{|\boldsymbol y|=1}\\
			=&\nabla\left[\tilde V^+_{m,\varepsilon}(\boldsymbol y)-\frac{\kappa_m^+}{2\pi}\ln\frac{1}{\varepsilon}\right]\bigg|_{|\boldsymbol y|=1}+O(\varepsilon)=s_{m,\varepsilon}^+\beta_{m,\varepsilon}^++O(\varepsilon).
		\end{align*}
		By applying the implicit function theorem, we have 
		$$\|t_{\varepsilon,\phi}\|_{C^1[0,2\pi)}=\frac{\| \tilde\phi_m^+\|_{C^1(B_L(\boldsymbol 0))}}{s_{m,\varepsilon}^+\beta_{m,\varepsilon}^++O(\varepsilon)}=O(\varepsilon),$$ 
		and 
		$$\|t_{\varepsilon,\mathcal F_m^+}\|_{C^1[0,2\pi)}=\frac{\|\mathcal F_m^+\|_{C^1(B_L(\boldsymbol 0))}}{s_{m,\varepsilon}^+\beta_{m,\varepsilon}^++O(\varepsilon)}=O(\varepsilon).$$
		Hence $\tilde{\mathbf\Gamma}_{m,\phi}^+(\xi)$ is a $C^1$ closed curve. The quantitative estimates \eqref{2-11} and \eqref{2-12} also follow from the implicit function theorem directly.
	\end{proof}
	
	By Lemma \ref{lem2-3}, we are to solve 
	\begin{equation*}
		\phi_\varepsilon=\mathcal T_\varepsilon N_\varepsilon(\phi_\varepsilon)
	\end{equation*}
	via contraction mapping theorem. In the following lemma, we verify the smallness and the contraction property for the nonlinear perturbation $N_\varepsilon(\phi_\varepsilon)$ and prove the existence of $\phi_\varepsilon$ in $E_\varepsilon$.
	\begin{lemma}\label{lem2-5}
		There exists a small $\varepsilon_0>0$ such that for any $\varepsilon\in(0,\varepsilon_0]$, there is a unique solution $\phi_\varepsilon\in E_\varepsilon$ to \eqref{2-4}. Moreover  $\phi_\varepsilon$ satisfies
		\begin{equation}\label{2-14}
			\|\phi_\varepsilon\|_*+\varepsilon^{1-\frac{2}{p}}\|\nabla\phi_\varepsilon\|_{L^p\left( \left(\cup_{m=1}^jB_{L\varepsilon}(\boldsymbol z_{m,\varepsilon}^+)\right)\cup \left(\cup_{n=1}^kB_{L\varepsilon}(\boldsymbol z_{n,\varepsilon}^-)\right)\right)}= O(\varepsilon)
		\end{equation}
		for $p\in(2,+\infty]$.
	\end{lemma}
	\begin{proof}
		Denote $\mathcal G_\varepsilon:=\mathcal T_\varepsilon R_\varepsilon$, and a neighborhood of origin in $E_\varepsilon$ as
		\begin{equation*}
			\mathcal B_\varepsilon:=E_\varepsilon\cap \left\{\phi \ | \ \|\phi\|_*+\varepsilon^{1-\frac{2}{p}}\|\nabla\phi\|_{L^p\left( \left(\cup_{m=1}^jB_{L\varepsilon}(\boldsymbol z_{m,\varepsilon}^+)\right)\cup \left(\cup_{n=1}^kB_{L\varepsilon}(\boldsymbol z_{n,\varepsilon}^-)\right)\right)}\le C\varepsilon, \ p\in(2,\infty]\right\}
		\end{equation*}
		with $C$ a large positive constant. We show that $\mathcal G_\varepsilon$ is a contraction map from $\mathcal B_\varepsilon$ to $\mathcal B_\varepsilon$ so that a unique fixed point $\phi_\varepsilon$ can be obtained by the contraction mapping theorem. Actually, letting $\mathbf h=N_\varepsilon(\phi)$ for $\phi\in\mathcal B_\varepsilon$, and noticing that $N_\varepsilon(\phi)$ satisfies the assumptions for $\mathbf h$ in Lemma \ref{lem2-3}, we obtain
		\begin{align*}
			\|\mathcal T_\varepsilon N_\varepsilon(\phi)\|_*+\varepsilon^{1-\frac{2}{p}}\|\nabla\mathcal T_\varepsilon N_\varepsilon(\phi)&\|_{L^p\left( \left(\cup_{m=1}^jB_{L\varepsilon}(\boldsymbol z_{m,\varepsilon}^+)\right)\cup \left(\cup_{n=1}^kB_{L\varepsilon}(\boldsymbol z_{n,\varepsilon}^-)\right)\right)}\\
			&\le c_0\varepsilon^{1-\frac{2}{p}}\|N_\varepsilon(\phi) \|_{W^{-1,p}\left( \left(\cup_{m=1}^jB_{L\varepsilon}(\boldsymbol z_{m,\varepsilon}^+)\right)\cup \left(\cup_{n=1}^kB_{L\varepsilon}(\boldsymbol z_{n,\varepsilon}^-)\right)\right)}.
		\end{align*}
		
		\bigskip
		
		To begin with, we are to show that $\mathcal G_\varepsilon$ maps $\mathcal B_\varepsilon$ continuously into itself. For simplicity, we only consider $N_\varepsilon(\phi_\varepsilon)$ restricted in $B_\delta(\boldsymbol z_{m,\ep}^+)$.  For each $\zeta(\boldsymbol y)\in C_0^\infty(B_{L}(\boldsymbol 0))$, by the estimate for $\Omega_{m,\varepsilon}^+$ in \ref{2-11}, we can calculate to obtain
		\begin{align*}
			&\quad\frac{s_{m,\varepsilon}^{+2}}{\varepsilon^2}\int_{B_L(\boldsymbol 0)}\left[\left(\tilde\psi_{m,\varepsilon}^++W\cos(s_{m,\varepsilon}^+ y_1+\theta_{m,\varepsilon}^+)-\mu_{m,\varepsilon}^+\right)_+^\gamma-\left(\tilde V_{m,\varepsilon}^+-\frac{\kappa_m^+}{2\pi}\ln\frac{1}{\varepsilon}\right)_+^\gamma\right]\zeta d\boldsymbol y\\
			&\quad\quad-\frac{\gamma s_{m,\varepsilon}^{+2}}{\varepsilon^2}\int_{B_L(\boldsymbol 0)}\left(\tilde V_{m,\varepsilon}^+-\frac{\kappa_m^+}{2\pi}\ln\frac{1}{\varepsilon}\right)_+^{\gamma-1}\phi \zeta d\boldsymbol y\\
			&=\frac{\gamma s_{m,\varepsilon}^{+2}}{\varepsilon^2}\int_{B_1(\boldsymbol 0)}\left(\tilde V_{m,\varepsilon}^+-\frac{\kappa_m^+}{2\pi}\ln\frac{1}{\varepsilon}\right)_+^{\gamma-1}\left[\mathcal F_m^++O(\ep^2)\right]\zeta d\boldsymbol y\\
			&\quad+\frac{s_{m,\varepsilon}^{+2}}{\varepsilon^2}\int_{B_{1+t_\ep}(\boldsymbol 0)\setminus B_1(\boldsymbol 0)}\left(\tilde\psi_{m,\varepsilon}^++W\cos(s_{m,\varepsilon}^+ y_1+\theta_{m,\varepsilon}^+)-\mu_{m,\varepsilon}^+\right)_+^\gamma\zeta d\boldsymbol y\\
			&=O(\varepsilon)\|\zeta\|_{W^{1,p'}(B_{L}(\boldsymbol 0)}.
		\end{align*}
		Hence it holds
		\begin{equation*}
			\varepsilon^{1-\frac{2}{p}}\|N_\varepsilon(\phi)\|_{W^{-1,p}\left( \left(\cup_{m=1}^jB_{L\varepsilon}(\boldsymbol z_{m,\varepsilon}^+)\right)\cup \left(\cup_{n=1}^kB_{L\varepsilon}(\boldsymbol z_{n,\varepsilon}^-)\right)\right)}=O(\varepsilon),
		\end{equation*}
		and
		\begin{equation*}
			\|\mathcal T_\varepsilon N_\varepsilon(\phi)\|_*+\varepsilon^{1-\frac{2}{p}}\|\nabla\mathcal T_\varepsilon N_\varepsilon(\phi)\|_{L^p\left( \left(\cup_{m=1}^jB_{L\varepsilon}(\boldsymbol z_{m,\varepsilon}^+)\right)\cup \left(\cup_{n=1}^kB_{L\varepsilon}(\boldsymbol z_{n,\varepsilon}^-)\right)\right)}=O(\varepsilon)
		\end{equation*}
		for $p\in(2,+\infty]$. Thus the operator $\mathcal G_\varepsilon$ indeed maps $\mathcal B_\varepsilon$ to $\mathcal B_\varepsilon$ continuously.

		In the next step, we are to verify that $\mathcal G_\varepsilon$ is a contraction mapping under the norm
		\begin{equation*}
			\|\cdot\|_{\mathcal G_\varepsilon}=\|\cdot\|_*+\varepsilon^{1-\frac{2}{p}}\|\cdot\|_{W^{1,p}\left( \left(\cup_{m=1}^jB_{L\varepsilon}(\boldsymbol z_{m,\varepsilon}^+)\right)\cup \left(\cup_{n=1}^kB_{L\varepsilon}(\boldsymbol z_{n,\varepsilon}^-)\right)\right)}, \quad p\in(2,+\infty].
		\end{equation*}
		We already know that $\mathcal B_\varepsilon$ is close to this norm. Let $\phi_1$ and $\phi_2$ be two functions in $\mathcal B_\varepsilon$. In view of Lemma \ref{lem2-3}, we have
		\begin{equation*}
			\|\mathcal G_\varepsilon\phi_1-\mathcal G_\varepsilon\phi_2\|_{\mathcal G_\varepsilon}\le C\varepsilon^{1-\frac{2}{p}}\|N_\varepsilon^+(\phi_1)-N_\varepsilon^+(\phi_2) \|_{W^{-1,p}\left( \left(\cup_{m=1}^jB_{L\varepsilon}(\boldsymbol z_{m,\varepsilon}^+)\right)\cup \left(\cup_{n=1}^kB_{L\varepsilon}(\boldsymbol z_{n,\varepsilon}^-)\right)\right)},
		\end{equation*}
		where
		\begin{align*}
			&\quad N_\varepsilon(\phi_1)-N_\varepsilon(\phi_2)\\
			&=\frac{1}{\varepsilon^2}\sum_{m=1}^j\boldsymbol 1_{B_\delta(\boldsymbol z_m^+)}\big[(\Psi_\varepsilon+\phi_1+W\cos\theta>\mu_{m,\varepsilon}^+)^\gamma_+-(\Psi_\varepsilon+\phi_2+W\cos\theta>\mu_{m,\varepsilon}^+)_+^\gamma\big]\\
			&\quad -\frac{\gamma}{\varepsilon^2}\sum_{m=1}^j\left(V^+_{m,\varepsilon}(\boldsymbol z)-\frac{\kappa_m^+}{2\pi}\ln\frac{1}{\varepsilon}\right)_+^{\gamma-1}(\phi_1-\phi_2)\\
			&\quad-\frac{1}{\varepsilon^2}\sum_{n=1}^k\boldsymbol 1_{B_\delta(\boldsymbol z_n^-)}\big[(-\Psi_\varepsilon-\phi_1-W\cos\theta>\mu_{n,\varepsilon}^-)^\gamma_+-(-\Psi_\varepsilon-\phi_2-W\cos\theta>\mu_{n,\varepsilon}^-)_+^\gamma\big]\\
			&\quad -\frac{\gamma}{\varepsilon^2}\sum_{n=1}^k\left(V^-_{n,\varepsilon}(\boldsymbol z)-\frac{\kappa_n^-}{2\pi}\ln\frac{1}{\varepsilon}\right)_+^{\gamma-1}(\phi_1-\phi_2)
		\end{align*}
		According to \eqref{2-11} in Lemma \ref{lem2-4}, and the Taylor expansion
		\begin{equation*}
			(1+t)^\gamma=1+\gamma t+O\left(t^{\min(\gamma,2)}\right), \quad t\ge 0, \ \gamma>1
		\end{equation*}
		for each $\zeta\in C_0^\infty(B_L(\boldsymbol 0))$, it holds
		\begin{align*}
			&\quad\,\, \frac{s_{m,\varepsilon}^{+2}}{\varepsilon^2}\int_{B_L(\boldsymbol 0)}\left[(\tilde\Psi_{m,\varepsilon}^++\tilde\phi_{m,1}^++W\cos(s_{m,\ep}^+y_1+\theta_{m,\ep}^+)>\mu_{m,\varepsilon}^+)^\gamma_+\right.\\
			&\left.\quad\quad-(\tilde\Psi_{m,\varepsilon}^++\tilde\phi_{m,2}^++W\cos(s_{m,\ep}^+y_1+\theta_{m,\ep}^+)>\mu_{m,\varepsilon}^+)_+^\gamma\right]\zeta d\boldsymbol y\\
			&=\frac{\gamma s_{m,\varepsilon}^{+2}}{\varepsilon^2}\int_{B_1(\boldsymbol 0)} \left(V^+_{m,\varepsilon}(\boldsymbol z)-\frac{\kappa_m^+}{2\pi}\ln\frac{1}{\varepsilon}\right)_+^{\gamma-1}\cdot(\phi_1-\phi_2)\zeta d\boldsymbol y\\
			&\quad\quad+\frac{s_{m,\varepsilon}^{+2}}{\varepsilon^2}\int_{B_{1+t_{\ep,\phi_1}}(\boldsymbol 0)\setminus B_{1+t_{\ep,\phi_2}}(\boldsymbol 0)}\left(\tilde\Psi_{m,\varepsilon}^++W\cos(s_{m,\varepsilon}^+ y_1+\theta_{m,\varepsilon}^+)-\mu_{m,\varepsilon}^+\right)_+^\gamma\zeta d\boldsymbol y\\
			&\quad\quad +O_\ep(1)\cdot\int_{B_1(\boldsymbol 0)}\left(\phi_1^{\min(\gamma-1,1)}-\phi_2^{\min(\gamma-1,1)}\right)(\phi_1-\phi_2)\zeta d\boldsymbol  y\\
			&=\frac{\gamma s_{m,\varepsilon}^{+2}}{\varepsilon^2}\int_{B_1(\boldsymbol 0)} \left(V^+_{m,\varepsilon}(\boldsymbol z)-\frac{\kappa_m^+}{2\pi}\ln\frac{1}{\varepsilon}\right)_+^{\gamma-1}\cdot(\phi_1-\phi_2)\zeta d\boldsymbol y\\
			&\quad\quad+o_\ep(1)\cdot\|\tilde\phi_1-\tilde\phi_2\|_{L^\infty(B_L(\boldsymbol 0))}\|\zeta\|_{W^{1,p'}(B_L(\boldsymbol 0))}.
		\end{align*}
		Thus we have
			\begin{align*}
			&\quad\,\, \frac{s_{m,\varepsilon}^{+2}}{\varepsilon^2}\int_{B_L(\boldsymbol 0)}\left[(\tilde\Psi_{m,\varepsilon}^++\tilde\phi_{m,1}^++W\cos(s_{m,\ep}^+y_1+\theta_{m,\ep}^+)>\mu_{m,\varepsilon}^+)^\gamma_+\right.\\
			&\left.\quad\quad-(\tilde\Psi_{m,\varepsilon}^++\tilde\phi_{m,2}^++W\cos(s_{m,\ep}^+y_1+\theta_{m,\ep}^+)>\mu_{m,\varepsilon}^+)_+^\gamma\right]\zeta d\boldsymbol y\\
			&\quad\quad-\frac{\gamma s_{m,\varepsilon}^{+2}}{\varepsilon^2}\int_{B_1(\boldsymbol 0)} \left(V^+_{m,\varepsilon}(\boldsymbol z)-\frac{\kappa_m^+}{2\pi}\ln\frac{1}{\varepsilon}\right)_+^{\gamma-1}\cdot(\phi_1-\phi_2)\zeta d\boldsymbol y\\
			&=o_\ep(1)\cdot\|\tilde\phi_1-\tilde\phi_2\|_{L^\infty(B_L(\boldsymbol 0))}||\zeta ||_{W^{1,p'}(B_L(\boldsymbol 0))}.
		\end{align*}
		By dealing with other vortices using the same procedure, we conclude that
		\begin{equation*}
			\varepsilon^{1-\frac{2}{p}}\|N_\varepsilon(\phi_1)-N_\varepsilon(\phi_2) \|_{W^{-1,p}\left( \left(\cup_{m=1}^jB_{L\varepsilon}(\boldsymbol z_{m,\varepsilon}^+)\right)\cup \left(\cup_{n=1}^kB_{L\varepsilon}(\boldsymbol z_{n,\varepsilon}^-)\right)\right)}=o_\varepsilon(1)\cdot \|\phi_1-\phi_2\|_{\mathcal G_\varepsilon},
		\end{equation*}
		and
		\begin{equation*}
			\|\mathcal G_\varepsilon\phi_1-\mathcal G_\varepsilon\phi_2\|_{\mathcal G_\varepsilon}=o_\varepsilon(1)\cdot \|\phi_1-\phi_2\|_{\mathcal G_\varepsilon}.
		\end{equation*}
		Hence we have shown that $\mathcal G_\varepsilon$ is a contraction map from $\mathcal B_\varepsilon$ into itself.
		
		By applying the contraction mapping theorem, we now can claim that there is a unique $\phi_\varepsilon\in \mathcal B_\varepsilon$ such that $\phi_\varepsilon=\mathcal G_\varepsilon\phi_\varepsilon$, which satisfies \eqref{2-14}. 
	\end{proof}
	
	Although we have the above existing result on projection problem \eqref{2-4}, the error function $\phi_\ep$ does not solve the primitive equation \eqref{2-2} yet. To obtain $\psi_\ep$ as solutions to \eqref{1-6}, we solve a $(2j+2k)$-dimensional problem in the kernel of $\mathbb L_\ep$.
	
	\smallskip
	
	\subsection{The $(2j+2k)$-dimensional problem}
	
	If we multiply \eqref{2-4} with $X_{l,\varepsilon}^\pm$, $Y_{i,\varepsilon}^\pm$ and integrating on $\mathbb S^2$,  we have following equations for the projection vector 
	$$\boldsymbol \Lambda=(a_1^+,\cdots,a_j^+,b_1^+,\cdots,b_j^+, a_1^-,\cdots,a_k^-,b_1^-,\cdots,b_k^-).$$
	\begin{footnotesize}
	\begin{align*}
		&\int_{\mathbb S^2}\left[(-\Delta_{\mathbb S^2})\psi_\varepsilon-\frac{1}{\varepsilon^2}\sum_{m=1}^j\boldsymbol1_{B_\delta(\boldsymbol z_m^+)}(\psi_\varepsilon+W\cos\theta-\mu_{m,\varepsilon}^+)_+^\gamma+\frac{1}{\varepsilon^2}\sum_{n=1}^k\boldsymbol1_{B_\delta(\boldsymbol z_n^-)}(-\psi_\varepsilon-W\cos\theta-\mu_{n,\varepsilon}^-)^\gamma_+\right]X_{m,\varepsilon}^+ d\boldsymbol\sigma\\
		&=a_m^+\int_{\mathbb S^2} X_{m,\varepsilon}^+(-\Delta_{\mathbb S^2})X_{m,\varepsilon}^+ d\boldsymbol \sigma,
	\end{align*}
	\begin{align*}
		&\int_{\mathbb S^2}\left[(-\Delta_{\mathbb S^2})\psi_\varepsilon-\frac{1}{\varepsilon^2}\sum_{m=1}^j\boldsymbol1_{B_\delta(\boldsymbol z_m^+)}(\psi_\varepsilon+W\cos\theta-\mu_{m,\varepsilon}^+)_+^\gamma+\frac{1}{\varepsilon^2}\sum_{n=1}^k\boldsymbol1_{B_\delta(\boldsymbol z_n^-)}(-\psi_\varepsilon-W\cos\theta-\mu_{n,\varepsilon}^-)^\gamma_+\right]Y_{m,\varepsilon}^+ d\boldsymbol\sigma\\
		&=b_m^+\int_{\mathbb S^2} Y_{m,\varepsilon}^+(-\Delta_{\mathbb S^2})Y_{m,\varepsilon}^+ d\boldsymbol \sigma,
	\end{align*}
	\begin{align*}
		&\int_{\mathbb S^2}\left[(-\Delta_{\mathbb S^2})\psi_\varepsilon-\frac{1}{\varepsilon^2}\sum_{m=1}^j\boldsymbol1_{B_\delta(\boldsymbol z_m^+)}(\psi_\varepsilon+W\cos\theta-\mu_{m,\varepsilon}^+)_+^\gamma+\frac{1}{\varepsilon^2}\sum_{n=1}^k\boldsymbol1_{B_\delta(\boldsymbol z_n^-)}(-\psi_\varepsilon-W\cos\theta-\mu_{n,\varepsilon}^-)^\gamma_+\right]X_{n,\varepsilon}^- d\boldsymbol\sigma\\
		&=a_n^-\int_{\mathbb S^2} X_{n,\varepsilon}^-(-\Delta_{\mathbb S^2})X_{n,\varepsilon}^- d\boldsymbol \sigma,
	\end{align*}
	\begin{align*}
		&\int_{\mathbb S^2}\left[(-\Delta_{\mathbb S^2})\psi_\varepsilon-\frac{1}{\varepsilon^2}\sum_{m=1}^j\boldsymbol1_{B_\delta(\boldsymbol z_m^+)}(\psi_\varepsilon+W\cos\theta-\mu_{m,\varepsilon}^+)_+^\gamma+\frac{1}{\varepsilon^2}\sum_{n=1}^k\boldsymbol1_{B_\delta(\boldsymbol z_n^-)}(-\psi_\varepsilon-W\cos\theta-\mu_{n,\varepsilon}^-)^\gamma_+\right]Y_{n,\varepsilon}^- d\boldsymbol\sigma\\
		&=b_n^-\int_{\mathbb S^2} Y_{n,\varepsilon}^-(-\Delta_{\mathbb S^2})Y_{n,\varepsilon}^- d\boldsymbol \sigma.
	\end{align*}
	\end{footnotesize}To make $\psi_\varepsilon$ a solution to \eqref{1-6}, we have the following lemma on the value of the left-hand side of the above equations. 
	\begin{lemma}\label{lem2-6}
		One has
		\begin{footnotesize}
		\begin{align*}
			&\quad\int_{\mathbb S^2}\left[(-\Delta_{\mathbb S^2})\psi_\varepsilon-\frac{1}{\varepsilon^2}\sum_{m=1}^j\boldsymbol1_{B_\delta(\boldsymbol z_m^+)}(\psi_\varepsilon+W\cos\theta-\mu_{m,\varepsilon}^+)_+^\gamma+\frac{1}{\varepsilon^2}\sum_{n=1}^k\boldsymbol1_{B_\delta(\boldsymbol z_n^-)}(-\psi_\varepsilon-W\cos\theta-\mu_{n,\varepsilon}^-)^\gamma_+\right]X_{m,\varepsilon}^+ d\boldsymbol\sigma\\
			&=C_m^+\cdot\left[\kappa_i^+\sum\limits_{i\neq m}^j\partial_\theta G(\boldsymbol z_{m,\varepsilon}^+,\boldsymbol z_{i,\varepsilon}^+)+ \kappa_m^+ \partial_\theta H(\boldsymbol z_{m,\varepsilon}^+,\boldsymbol z_{m,\varepsilon}^+)-\kappa_l^-\sum\limits_{l=1}^k\partial_\theta G(\boldsymbol z_{m,\varepsilon}^+,\boldsymbol z_{l,\varepsilon}^-)-W\sin\theta_{m,\varepsilon}^++o_\varepsilon(1)\right],
		\end{align*}
		\begin{align*}
			&\quad\int_{\mathbb S^2}\left[(-\Delta_{\mathbb S^2})\psi_\varepsilon-\frac{1}{\varepsilon^2}\sum_{m=1}^j\boldsymbol1_{B_\delta(\boldsymbol z_m^+)}(\psi_\varepsilon+W\cos\theta-\mu_{m,\varepsilon}^+)_+^\gamma+\frac{1}{\varepsilon^2}\sum_{n=1}^k\boldsymbol1_{B_\delta(\boldsymbol z_n^-)}(-\psi_\varepsilon-W\cos\theta-\mu_{n,\varepsilon}^-)^\gamma_+\right]Y_{m,\varepsilon}^+ d\boldsymbol\sigma\\
			&=C_m^-\cdot\left[\kappa_i^+\sum\limits_{i\neq m}^j\partial_\varphi G(\boldsymbol z_{m,\varepsilon}^+,\boldsymbol z_{i,\varepsilon}^+)+ \kappa_m^+ \partial_\varphi H(\boldsymbol z_{m,\varepsilon}^+,\boldsymbol z_{m,\varepsilon}^+)-\kappa_l^-\sum\limits_{l=1}^k\partial_\varphi G(\boldsymbol z_{m,\varepsilon}^+,\boldsymbol z_{l,\varepsilon}^-)+o_\varepsilon(1)\right],
		\end{align*}
		\begin{align*}
			&\quad\int_{\mathbb S^2}\left[(-\Delta_{\mathbb S^2})\psi_\varepsilon-\frac{1}{\varepsilon^2}\sum_{m=1}^j\boldsymbol1_{B_\delta(\boldsymbol z_m^+)}(\psi_\varepsilon+W\cos\theta-\mu_{m,\varepsilon}^+)_+^\gamma+\frac{1}{\varepsilon^2}\sum_{n=1}^k\boldsymbol1_{B_\delta(\boldsymbol z_n^-)}(-\psi_\varepsilon-W\cos\theta-\mu_{n,\varepsilon}^-)^\gamma_+\right]X_{n,\varepsilon}^- d\boldsymbol\sigma\\
			&=C_n^+\cdot\left[\kappa_l^+\sum\limits_{l\neq n}^k\partial_\theta G(\boldsymbol z_{n,\varepsilon}^-,\boldsymbol z_{l,\varepsilon}^-)+ \kappa_n^+ \partial_\theta H(\boldsymbol z_{n,\varepsilon}^+,\boldsymbol z_{n,\varepsilon}^+)-\kappa_i^-\sum\limits_{i=1}^j\partial_\theta G(\boldsymbol z_{n,\varepsilon}^+,\boldsymbol z_{i,\varepsilon}^-)-W\sin\theta_{n,\varepsilon}^-+o_\varepsilon(1)\right],
		\end{align*}
		\begin{align*}
			&\quad\int_{\mathbb S^2}\left[(-\Delta_{\mathbb S^2})\psi_\varepsilon-\frac{1}{\varepsilon^2}\sum_{m=1}^j\boldsymbol1_{B_\delta(\boldsymbol z_m^+)}(\psi_\varepsilon+W\cos\theta-\mu_{m,\varepsilon}^+)_+^\gamma+\frac{1}{\varepsilon^2}\sum_{n=1}^k\boldsymbol1_{B_\delta(\boldsymbol z_n^-)}(-\psi_\varepsilon-W\cos\theta-\mu_{n,\varepsilon}^-)^\gamma_+\right]Y_{n,\varepsilon}^- d\boldsymbol\sigma\\
			&=C_n^-\cdot\left[\kappa_l^+\sum\limits_{l\neq n}^k\partial_\varphi G(\boldsymbol z_{n,\varepsilon}^-,\boldsymbol z_{l,\varepsilon}^-)+ \kappa_n^+ \partial_\varphi H(\boldsymbol z_{n,\varepsilon}^+,\boldsymbol z_{n,\varepsilon}^+)-\kappa_i^-\sum\limits_{i=1}^j\partial_\varphi G(\boldsymbol z_{n,\varepsilon}^+,\boldsymbol z_{i,\varepsilon}^-)+o_\varepsilon(1)\right],
		\end{align*}
		\end{footnotesize}
		where $C_m^\pm$ and $C_m^\pm$ are positive constants.
	\end{lemma}
		\begin{proof}
		Notice that
		\begin{align*}
			(-\Delta_{\mathbb S^2})\psi_\varepsilon-\frac{1}{\varepsilon^2}\sum_{m=1}^j\boldsymbol1_{B_\delta(\boldsymbol z_m^+)}(\psi_\varepsilon+W\cos\theta-\mu_{m,\varepsilon}^+)_+^\gamma+\frac{1}{\varepsilon^2}\sum_{n=1}^k\boldsymbol1_{B_\delta(\boldsymbol z_n^-)}(-\psi_\varepsilon-W\cos\theta-\mu_{n,\varepsilon}^-)^\gamma_+\\
			=\mathbb L_\varepsilon\phi_\varepsilon-N_\varepsilon(\phi_\varepsilon).
		\end{align*}
		For simplicity, we consider the following integration for the $m$th positive vortex in $B_\delta(\boldsymbol z_m^+)$.
		\begin{equation*}
			\int_{B_\delta(\boldsymbol z_m^+)}X_{m,\varepsilon}^+\mathbb L_\varepsilon\phi_\varepsilon d\boldsymbol \sigma-\int_{B_\delta(\boldsymbol z_m^+)}X_{m,\ep}^+ N_\varepsilon(\phi_\varepsilon)d\boldsymbol \sigma=I_1-I_2.
		\end{equation*}
		
		For the first part of the integration $I_1$, from Lemma \ref{lem2-3}, we already know that 
		$$I_1=\frac{C}{\varepsilon|\ln\varepsilon|}\|\phi_\varepsilon\|_*=o_\varepsilon(1).$$ 
		For the second part $I_2$, we can use integration by parts to obtain  
		\begin{align*}
			&\int_{B_\delta(\boldsymbol z_m^+)}X_{m,\varepsilon}^+ N_\varepsilon(\phi_\varepsilon)d\boldsymbol \sigma=\frac{1}{\varepsilon^2}\left[\boldsymbol1_{B_\delta(\boldsymbol z_m^+)}\left(\psi_\varepsilon+W_\varepsilon\cos\theta-\mu_{m,\varepsilon}^+\right)_+^\gamma-\left(V^+_{m,\varepsilon}(\boldsymbol z)-\frac{\kappa_m^+}{2\pi}\ln\frac{1}{\varepsilon}\right)_+^\gamma\right.\\
			&\left. \quad-\gamma\left(V^+_{m,\varepsilon}(\boldsymbol z)-\frac{\kappa_m^+}{2\pi}\ln\frac{1}{\varepsilon}\right)_+^{\gamma-1}\phi_\varepsilon\right]X_{m,\varepsilon}^+ d\boldsymbol\sigma\\
			&=\frac{s_{m,\varepsilon}^{+}}{\varepsilon^2}\int_{B_1(\boldsymbol 0)}\gamma (w_\gamma(\boldsymbol y))_+^{\gamma-1}\left[\frac{\mathcal F_m^+}{s_{m,\varepsilon}^+\beta_{m,\varepsilon}^+} s_\ep y_1+O(\varepsilon^2)\right]\partial_{y_1} w_\gamma(\boldsymbol y) d\boldsymbol y+o_\varepsilon(1)\\
			&=-\frac{s_{m,\varepsilon}^{+}}{\varepsilon^2\beta_{m,\varepsilon}^+}\mathcal F_m^+\int_{B_1(\boldsymbol 0)} (w_\gamma(\boldsymbol y))_+^{\gamma}d\boldsymbol y+o_\varepsilon(1).
		\end{align*}
		According to the definition of $\beta_{m,\varepsilon}^+$ in \eqref{2-2} and $\mathcal F_m^+$ in Lemma \ref{lem2-4}, we have 
		\begin{align*}
			I_2=-C_m^+\left[\sum\limits_{l\neq m}^j\kappa_l^+\partial_\theta G(\boldsymbol z_{m,\ep}^+,\boldsymbol z_{l,\ep}^+)+ \kappa_m^+ \partial_\theta H(\boldsymbol z_{m,\ep}^+,\boldsymbol z_{m,\ep}^+)-\sum\limits_{n=1}^k\kappa_n^-\partial_\theta G(\boldsymbol z_{m,\ep}^+,\boldsymbol z_{n,\ep}^-)-W\sin\theta_{m,\ep}^+\right].
		\end{align*}
		
		Since the integration for other vortices is similar to the $m$th positive vortex, we complete the proof of this lemma.
	\end{proof}
	
	\smallskip
	
	\noindent{\bf Proof of Theorem \ref{thm1}:} To obtain a family of desired solutions $\psi_\ep$ to \eqref{1-6}, we need to choose suitable locations $\boldsymbol z_{m,\varepsilon}^+, \boldsymbol z_{n,\varepsilon}^-$ such that $\mathbf \Lambda=\boldsymbol 0$. Notice that 
	\begin{equation*}
		\int_{\mathbb S^2} X_{l,\varepsilon}^\pm(-\Delta_{\mathbb S^2})X_{l,\varepsilon}^\pm d\boldsymbol \sigma>0, \quad \mathrm{and} \quad \int_{\mathbb S^2} Y_{i,\varepsilon}^\pm(-\Delta_{\mathbb S^2})Y_{i,\varepsilon}^\pm d\boldsymbol \sigma>0.
	\end{equation*} 
	Hence by Lemma \ref{lem2-6} and $(\boldsymbol z_1^+,\cdots,\boldsymbol z_j^+, \boldsymbol z_1^-,\cdots, \boldsymbol z_k^-)$ being a nondegenerate critical point of Kirchhoff--Routh function $\mathcal K_{k+j}$ in \eqref{1-5}, we claim that there exists a proper location series
	\begin{equation*}
		\left(\boldsymbol z_{1,\varepsilon}^+,\cdots,\boldsymbol z_{j,\varepsilon}^+, \boldsymbol z_{1,\varepsilon}^-,\cdots, \boldsymbol z_{k,\varepsilon}^-\right)= \left(\boldsymbol z_1^+,\cdots,\boldsymbol z_j^+, \boldsymbol z_1^-,\cdots, \boldsymbol z_k^-\right)+o_\varepsilon(1)
	\end{equation*}
	such that $\mathbf \Lambda=\boldsymbol 0$, which gives the existence and limiting behavior for vortex centers 
	$$(\boldsymbol z_{1,\varepsilon}^+,\cdots,\boldsymbol z_{j,\varepsilon}^+, \boldsymbol z_{1,\varepsilon}^-,\cdots, \boldsymbol z_{k,\varepsilon}^-)$$
	in (ii). While the estimates for vorticity sets $\Omega_{m,\varepsilon}^+$ and $\Omega_{n,\varepsilon}^-$ are from Lemma \ref{lem2-4}, and the radius $s_{l,\varepsilon}^\pm$ for each vortex set satisfies
	$$s_{l,\varepsilon}^\pm=\left(\frac{2\pi w'_\gamma(1)}{\kappa_l^\pm}\right)^{\frac{\gamma-1}{2}} \varepsilon+o(\ep)$$
	by \eqref{2-1}.  Moreover, since $\gamma>1$, the vorticity function $\omega_\ep$ is $C^1$ smooth by the regularity theory of elliptic operator, and $\ep^2\|\phi_\ep\|_{C^2(B_L(\boldsymbol 0))}=O(\ep)$ by bootstrap argument. Then using the implicit functional theorem, we can verify that
	$$\|t_{\varepsilon}\|_{C^2[0,2\pi)}=\frac{\| \tilde\phi_l^\pm+\mathcal F_l^\pm\|_{C^2(B_L(\boldsymbol 0))}}{s_{l,\varepsilon}^\pm\beta_{l,\varepsilon}^\pm+O(\varepsilon)}=O(\varepsilon),$$
	and the curvature of $\partial\Omega_{l,\ep}^\pm$ is
	$$\boldsymbol \kappa_{l,\ep}^\pm= \frac{(1+t_\ep)^2+2t_\ep'^2-(1+t_\ep)t_\ep''}{s_{l,\ep}^\pm[(1+t_\ep)^2+t_\ep''^2]^{\frac{3}{2}}}=\frac{1+O(\ep)}{s_{l,\ep}^\pm}>0,$$
	from which we can claim that $\partial\Omega_{l,\ep}^\pm$ is a $C^2$ convex curve. According to our construction, the convergence for $\psi_\ep$ and $\omega_\varepsilon=(-\Delta_{\mathbb S^2})\psi_\varepsilon$ stated in (i) (ii) is obvious. Hence we have completed the proof.	\qed
	
	\bigskip
	
	\section{Regularization for the case $\gamma=1$}\label{sec3}
	
	For $\gamma=1$, we let $\tau>0$ be the constant such that $1$ is the first eigenvalue of $-\Delta$ in $B_\tau(\boldsymbol 0)$ with the zero Dirichlet boundary condition, and the radial decreasing function $w_1(\boldsymbol y)> 0$ be the first eigenfunction for $-\Delta$ in $B_\tau(\boldsymbol 0)$ with $w_1(\boldsymbol 0)=1$. Then $w_1$ gives the ground state for the classic plasma problem $-\Delta w_1(\boldsymbol y)=(w_1)_+$.
	
	We can approximate the singular part in the Green function by
	\begin{equation*}
		V_{l,\ep}^\pm(\boldsymbol x)=\begin{cases} \frac{\kappa_l^\pm}{2\pi}\ln \frac{1}{ \ep}+\frac{\kappa_l^\pm}{\kappa} \cdot w_1\left(\frac{|A(\boldsymbol z-\boldsymbol z_{l,\ep}^\pm)|}{\ep}\right),\quad &|A(\boldsymbol z-\boldsymbol z_{l,\ep}^\pm)|\le \tau \ep ,\\ 
			\frac{\kappa_l^\pm}{2\pi} \ln\left(\frac{\tau}{|A(\boldsymbol z-\boldsymbol z_{l,\ep}^\pm)|}\right),&|A(\boldsymbol z-\boldsymbol z_{l,\ep}^\pm)|\ge  \tau \ep,\end{cases}
	\end{equation*}
	with $\kappa=\int_{\mathbb R^2} (w_1(\boldsymbol y))_+d\boldsymbol y$, and let
	\begin{equation*}
		R^\pm_{l,\varepsilon}(\boldsymbol z)=\frac{1}{\varepsilon^2}\int_{\{V^\pm_{l,\varepsilon}(\boldsymbol z)>\frac{\kappa_l^\pm}{2\pi}\ln\frac{1}{\varepsilon}\}} H(\theta,\varphi,\theta',\varphi')\left(V^\pm_{l,\varepsilon}-\frac{\kappa_l^\pm}{2\pi}\ln\frac{1}{\varepsilon}\right)_+ d\boldsymbol \sigma(\boldsymbol z').
	\end{equation*} 
	According to the definition of $V^\pm_{i,\varepsilon}(\boldsymbol z)$ and $R^\pm_{i,\varepsilon}(\boldsymbol z)$, it holds 
	\begin{equation*}
		(-\Delta_{\mathbb S^2})\left(V^\pm_{l,\varepsilon}+R^\pm_{l,\varepsilon}\right)=\frac{1}{\varepsilon^2}\left(V^\pm_{l,\varepsilon}(\boldsymbol z)-\frac{\kappa_l^\pm}{2\pi}\ln\frac{1}{\varepsilon}\right)_+.
	\end{equation*}
	Then by the decomposition
	\begin{align*}
		\psi_\varepsilon(\boldsymbol z)&=\sum\limits_{m=1}^jV^+_{m,\varepsilon}+\sum\limits_{m=1}^jR^+_{m,\varepsilon}-\sum\limits_{n=1}^kV^-_{n,\varepsilon}-\sum\limits_{n=1}^kR^-_{n,\varepsilon}+\phi_\varepsilon\\
		&:= \Psi_\varepsilon+\phi_\varepsilon,
	\end{align*}
	we can linearize the equation \eqref{1-6} with $\gamma=1$ as
	\begin{align*}
		0&=-\varepsilon^2\Delta_{\mathbb S^2}\left(\sum\limits_{m=1}^jV^+_{i,\varepsilon}+\sum\limits_{m=1}^jR^+_{i,\varepsilon}-\sum\limits_{n=1}^kV^-_{i,\varepsilon}-\sum\limits_{n=1}^kR^-_{i,\varepsilon}+\phi_\varepsilon\right)\\
		&\quad-\sum\limits_{m=1}^j\boldsymbol1_{B_\delta(\boldsymbol z_m^+)}\left(\psi_\varepsilon+W_\varepsilon\cos\theta-\mu_{m,\varepsilon}^+\right)_++\sum\limits_{n=1}^k\boldsymbol1_{ B_\delta(\boldsymbol z_n^-)}\left(-\psi_\varepsilon-W_\varepsilon\cos\theta-\mu_{n,\varepsilon}^-\right)_+\\
		&=\sum\limits_{m=1}^j\left(-\varepsilon^2\Delta_{\mathbb S^2}\left(V^+_{m,\varepsilon}+R^+_{m,\varepsilon}\right)-\left(V^+_{m,\varepsilon}(\boldsymbol z)-\frac{\kappa_m^+}{2\pi}\ln\frac{1}{\varepsilon}\right)_+\right)\\
		&\quad -\sum\limits_{n=1}^k\left(-\varepsilon^2\Delta_{\mathbb S^2}\left(V^-_{n,\varepsilon}+R^-_{n,\varepsilon}\right)-\left(V^-_{n,\varepsilon}(\boldsymbol z)-\frac{\kappa_n^-}{2\pi}\ln\frac{1}{\varepsilon}\right)_+\right)\\
		& \ \ \ +\left(-\varepsilon^2\Delta_{\mathbb S^2}\phi_\varepsilon-\sum\limits_{m=1}^j \boldsymbol 1_{B_{s_{m,\ep}^+}(\boldsymbol z_{m,\ep}^+)}\phi_\varepsilon-\sum\limits_{n=1}^k\boldsymbol 1_{B_{s_{n,\ep}^-}(\boldsymbol z_{n,\ep}^-)}\phi_\varepsilon\right)\\
		& \ \ \ -\sum\limits_{m=1}^j\left(\boldsymbol1_{B_\delta(\boldsymbol z_m^+)}\left(\psi_\varepsilon+W_\varepsilon\cos\theta-\mu_{m,\varepsilon}^+\right)_+-\left(V^+_{m,\varepsilon}(\boldsymbol z)-\frac{\kappa_m^+}{2\pi}\ln\frac{1}{\varepsilon}\right)_+-\boldsymbol 1_{B_{s_{m,\ep}^+}(\boldsymbol z_{m,\ep}^+)}\phi_\varepsilon\right)\\
		& \ \ \ +\sum\limits_{n=1}^k\left(\boldsymbol1_{ B_\delta(\boldsymbol z_n^-)}\left(-\psi_\varepsilon-W_\varepsilon\cos\theta-\mu_{n,\varepsilon}^-\right)_+-\left(V^-_{n,\varepsilon}(\boldsymbol z)-\frac{\kappa_n^-}{2\pi}\ln\frac{1}{\varepsilon}\right)_++\boldsymbol 1_{B_{s_{n,\ep}^-}(\boldsymbol z_{n,\ep}^-)}\phi_\varepsilon\right)\\
		&=\varepsilon^2\mathbb L_\varepsilon\phi_\varepsilon-\varepsilon^2N_\varepsilon(\phi_\varepsilon),
	\end{align*}
	where
	\begin{equation*}
		\mathbb L_\varepsilon\phi_\varepsilon:=(-\Delta_{\mathbb S^2})\phi_\varepsilon-\frac{1}{\ep^2}\sum\limits_{m=1}^j \boldsymbol 1_{B_{s_{m,\ep}^+}(\boldsymbol z_{m,\ep}^+)}\phi_\varepsilon-\frac{1}{\ep^2}\sum\limits_{n=1}^k\boldsymbol 1_{B_{s_{n,\ep}^-}(\boldsymbol z_{n,\ep}^-)}\phi_\varepsilon
	\end{equation*}
	is the linear term, and
	\begin{footnotesize}
	\begin{align*}
		N_\varepsilon(\phi_\varepsilon)=&\frac{1}{\ep^2}\sum\limits_{m=1}^j\left(\boldsymbol1_{B_\delta(\boldsymbol z_m^+)}\left(\psi_\varepsilon+W_\varepsilon\cos\theta-\mu_{m,\varepsilon}^+\right)_+-\left(V^+_{m,\varepsilon}(\boldsymbol z)-\frac{\kappa_m^+}{2\pi}\ln\frac{1}{\varepsilon}\right)_+-\boldsymbol 1_{B_{s_{m,\ep}^+}(\boldsymbol z_{m,\ep}^+)}\phi_\varepsilon\right)\\
		&-\frac{1}{\ep^2}\sum\limits_{n=1}^k\left(\boldsymbol1_{ B_\delta(\boldsymbol z_n^-)}\left(-\psi_\varepsilon-W_\varepsilon\cos\theta-\mu_{n,\varepsilon}^-\right)_+-\left(V^-_{n,\varepsilon}(\boldsymbol z)-\frac{\kappa_n^-}{2\pi}\ln\frac{1}{\varepsilon}\right)_++\boldsymbol 1_{B_{s_{n,\ep}^-}(\boldsymbol z_{n,\ep}^-)}\phi_\varepsilon\right)
	\end{align*}
	\end{footnotesize}is the nonlinear perturbation. To make $N_\varepsilon(\phi_\varepsilon)$ sufficiently small, the flux constant $\mu_{m,\varepsilon}^+$ and $\mu_{n,\varepsilon}^-$ are chosen the same way as in Section \ref{sec2}. The next step is to deal with the semi-linear problem
	\begin{equation*}
		\mathbb L_\varepsilon\phi_\varepsilon=N_\varepsilon(\phi_\varepsilon),
	\end{equation*}
	where a similar difficulty induced by a non-empty kernel of $\mathbb L_\varepsilon$ as the previous case $\gamma>1$ arises. Fortunately, for the linearized operator
	\begin{equation}\label{3-1}
		\mathbb L_0\phi:=-\Delta\phi-\boldsymbol 1_{B_\tau(\boldsymbol 0)}\phi
	\end{equation}
	of classic plasma problem, we have the following theorem by Flucher and Wei \cite{FW}.
	\begin{theorem}\label{thm4}
		Let $v\in L^\infty(\mathbb R^2)\cap C(\mathbb R^2)$ be a solution to $\mathbb L _0v=0$ with the operator $\mathbb L_0$ defined in \eqref{3-1}. Then
		\begin{equation*}
			v\in\mathrm{span} \left\{\frac{\partial w_1}{\partial y_1},\frac{\partial w_1}{\partial y_2}\right\}.
		\end{equation*}
	\end{theorem}
	Thus we can define the approximate kernel for $\mathbb L_\ep$ the same way as in Section \ref{sec2}, which is spanned by $X_{m,\varepsilon}^+({\boldsymbol z}), Y_{m,\varepsilon}^+({\boldsymbol z})$ with $1\le m\le j$, and $X_{n,\varepsilon}^-({\boldsymbol z}), Y_{n,\varepsilon}^-({\boldsymbol z})$ with $1\le n\le k$. The projection problem for the case $\gamma=1$ is 
		\begin{equation}\label{3-2}
		\begin{cases}
			\mathbb L_\varepsilon\phi=\mathbf h(\boldsymbol z)+(-\Delta_{\mathbb S^2})\sum_{m=1}^j\left[a_m^+X_{m,\varepsilon}^++b_m^+Y_{m,\varepsilon}^+\right]\\
			\quad\quad\quad\quad\quad\ +(-\Delta_{\mathbb S^2})\sum_{n=1}^k\left[a_n^-X_{n,\varepsilon}^-+b_n^-Y_{n,\varepsilon}^-\right], \ \ &\text{in} \ \mathbb S^2,\\
			\int_{\mathbb S^2}  \phi(\boldsymbol z)(-\Delta_{\mathbb S^2})X_{l,\varepsilon}^\pm(\boldsymbol z) d\boldsymbol \sigma=0, \quad
			\int_{\mathbb S^2}  \phi(\boldsymbol z)(-\Delta_{\mathbb S^2})Y_{i,\varepsilon}^\pm(\boldsymbol z) d\boldsymbol \sigma=0,
		\end{cases}
	\end{equation}
	where $\mathbf h(\boldsymbol z)$ satisfies 
	$$\mathrm{supp}\, \mathbf h(\boldsymbol z)\subset \left(\cup_{m=1}^jB_{L\varepsilon}(\boldsymbol z_{m,\varepsilon}^+)\right)\cup \left(\cup_{n=1}^kB_{L\varepsilon}(\boldsymbol z_{n,\varepsilon}^-)\right)$$
	with $L$ a large positive constant. As in Section \ref{sec2}, we let 
	$$\mathbf\Lambda=(a_1^+,\cdots,a_j^+,b_1^+,\cdots,b_j^+, a_1^-,\cdots,a_k^-,b_1^-,\cdots,b_k^-)$$
	be the $(2j+2k)$-dimensional projection vector determined by
	\begin{footnotesize}
		\begin{equation*}
			a_l^\pm\int_{\mathbb S^2} X_{l,\varepsilon}^\pm(-\Delta_{\mathbb S^2})X_{l,\varepsilon}^\pm d\boldsymbol \sigma=\int_{\mathbb S^2}X_{l,\varepsilon}^\pm\big[\mathbb L_\ep\phi-\mathbf h(\boldsymbol z)\big]d\boldsymbol \sigma, \quad b_l^\pm\int_{\mathbb S^2} Y_{l,\varepsilon}^\pm(-\Delta_{\mathbb S^2})Y_{l,\varepsilon}^\pm d\boldsymbol \sigma=\int_{\mathbb S^2}Y_{l,\varepsilon}^\pm\big[\mathbb L_\ep\phi-\mathbf h(\boldsymbol z)\big]d\boldsymbol \sigma.
		\end{equation*}
	\end{footnotesize}Then the solvability of \eqref{3-2} is a analogy of Lemma \ref{lem2-3}, and we claim that for $\ep>0$ small and $\mathbf\Lambda$ the projection vector, \eqref{3-2} has a unique solution $\phi_\varepsilon=\mathcal T_\varepsilon \, \mathbf h$, such that
	\begin{equation}\label{3-}
		\begin{split}
			\|\phi_\varepsilon\|_*+\varepsilon^{1-\frac{2}{p}}&\|\nabla\phi_\varepsilon\|_{L^p(\left(\cup_{m=1}^jB_{L\varepsilon}(\boldsymbol z_{m,\varepsilon}^+)\right)\cup \left(\cup_{n=1}^kB_{L\varepsilon}(\boldsymbol z_{n,\varepsilon}^-)\right))}\\
			&\le c_0\varepsilon^{1-\frac{2}{p}}\|\mathbf h\|_{W^{-1,p}(\left(\cup_{m=1}^jB_{L\varepsilon}(\boldsymbol z_{m,\varepsilon}^+)\right)\cup \left(\cup_{n=1}^kB_{L\varepsilon}(\boldsymbol z_{n,\varepsilon}^-)\right))},
		\end{split}
	\end{equation}
	with $\|\phi\|_*=\sup_{\mathbb S^2} |\phi|$. Let $\mathbf h(\boldsymbol z)=N_\ep(\phi_\ep)$. We are to estimate $N_\ep(\phi_\ep)$ so that the contraction mapping theorem can be applied to $\phi_\varepsilon=\mathcal T_\varepsilon N_\ep(\phi_\ep)$ and obtain the existence of $\phi_\ep$.
	
	\smallskip
	
	Denote  
	$$\tilde v_m^+(\boldsymbol y)=v(\ep y_1+\theta_{m,\ep}^+,\ep\sin^{-1}(\ep y_1+\theta_{m,\ep}^+)y_2+\varphi_{m,\ep}^+),$$
	and 
	$$\tilde v_n^-(\boldsymbol y)=v(\ep y_1+\theta_{n,\ep}^-,\ep\sin^{-1}(\ep y_1+\theta_{n,\ep}^-)y_2+\varphi_{n,\ep}^-)$$
	for a general function $v$. Using the fact
	\begin{align*}
		&\nabla\left[\tilde\psi_{l,\varepsilon}^\pm(\boldsymbol y)-W\cos(\ep y_1+\theta_{l,\ep}^\pm)-\mu_{l,\varepsilon}^\pm\right]\bigg|_{|\boldsymbol y|=\tau}\\
		=&\nabla\left[\tilde V^\pm_{l,\varepsilon}(\boldsymbol y)-\frac{\kappa_l^\pm}{2\pi}\ln\frac{1}{\varepsilon}\right]\bigg|_{|\boldsymbol y|=\tau}+O(\varepsilon)=\frac{\kappa_l^\pm}{\kappa w_1'(\tau)}+O(\varepsilon),
	\end{align*}
	and the implicit function theorem, we have the following lemma concerning the estimate for level sets 
	$$\{\boldsymbol z\in B_\delta(\boldsymbol z_{m}^+)\mid \psi_\varepsilon(\boldsymbol z)+W\sin\theta_0=\mu_{m,\varepsilon}^+\} \quad \mathrm{and} \quad \{\boldsymbol z\in B_\delta(\boldsymbol z_{n}^-)\mid -\psi_\varepsilon(\boldsymbol z)-W\sin\theta_0=\mu_{n,\varepsilon}^-\}$$
	as the perturbations of $B_{s_{m,\ep}^+}(\boldsymbol z_{m,\ep}^+)$ and $B_{s_{n,\ep}^-}(\boldsymbol z_{m,\ep}^-)$.
	\begin{lemma}\label{lem3-2}
		Suppose that $\phi$ is a function satisfying
		\begin{equation}\label{3-4}
			\|\phi_\varepsilon\|_*+\varepsilon^{1-\frac{2}{p}}\|\nabla\phi_\varepsilon\|_{L^p\left(\left(\cup_{m=1}^jB_{L\varepsilon}(\boldsymbol z_{m,\varepsilon}^+)\right)\cup \left(\cup_{n=1}^kB_{L\varepsilon}(\boldsymbol z_{n,\varepsilon}^-)\right)\right)}= O(\varepsilon)
		\end{equation}
		with $p\in(2,+\infty]$. Then the sets
		$$\tilde{\mathbf\Gamma}_{m,\phi}^+:=\{\boldsymbol y \mid \tilde\Psi_{m,\varepsilon}^+(\boldsymbol y)+\tilde\phi_m^++W\cos(\ep y_1+\theta_{m,\phi}^+)=\mu_{m,\varepsilon}^+\}$$
		and
		$$\tilde{\mathbf\Gamma}_{n,\phi}^-:=\{\boldsymbol y \mid -\tilde\Psi_{n,\varepsilon}^-(\boldsymbol y)-\tilde\phi_n^--W\cos(\ep y_1+\theta_{n,\varepsilon}^-)=\mu_{n,\varepsilon}^-\}$$
		are $C^1$ closed curves in $\mathbb{R}^2$, and
		\begin{equation*}
			\begin{split}
				\tilde{\mathbf\Gamma}_{m,\phi}^+(\xi)&=(\tau+t_\varepsilon(\xi))(\cos\xi,\sin\xi)\\
				&=(\tau+t_{\varepsilon,\phi}(\xi)+t_{\varepsilon,\mathcal F_m^+}(\xi)+O(\varepsilon^2))(\cos\xi,\sin\xi)\\
				&=\left(\tau+\frac{\kappa}{\kappa_m^+w_1'(\tau)}\tilde\phi(\cos\xi,\sin\xi)\right)(\cos\xi,\sin\xi)+\frac{\kappa\mathcal F^+_m(\xi)}{\kappa_m^+w_1'(\tau)}(\cos\theta,0)\\
				&\quad+O(\varepsilon^2), \quad \xi\in (0,2\pi],
			\end{split}
		\end{equation*}
		\begin{equation*}
			\begin{split}
				\tilde{\mathbf\Gamma}_{n,\phi}^-(\xi)&=(\tau+t_\varepsilon(\xi))(\cos\xi,\sin\xi)\\
				&=(\tau+t_{\varepsilon,\phi}(\xi)+t_{\varepsilon,\mathcal F_m^-}(\xi)+O(\varepsilon^2))(\cos\xi,\sin\xi)\\
				&=\left(\tau+\frac{\kappa}{\kappa_n^-w_1'(\tau)}\tilde\phi(\cos\xi,\sin\xi)\right)(\cos\xi,\sin\xi)+\frac{\kappa\mathcal F^-_n(\xi)}{\kappa_n^-w_1'(\tau)}(\cos\theta,0)\\
				&\quad+O(\varepsilon^2), \quad \xi\in (0,2\pi]
			\end{split}
		\end{equation*}
		with $\|t_\varepsilon\|_{C^1[0,2\pi)}=O(\ep)$, and 
		\begin{align*}
			\mathcal F_m^+(\xi)&=\bigg\langle \sum\limits_{l\neq m}^j\kappa_{l}^+\nabla G(\boldsymbol z_{m,\ep}^+,\boldsymbol z^+_{l,\ep})+\kappa_m^+ \nabla H(\boldsymbol z_{m,\ep}^+,\boldsymbol z^+_{m,\ep})-\sum\limits_{n=1}^k\kappa_n^-\nabla G(\boldsymbol z_{m,\ep}^+,\boldsymbol z^-_{l,\ep})\\
			&\quad-(W\sin\theta_{m,\ep}^+,0),\quad (\ep\cos\xi,\ep\sin^{-1}(\ep y_1+\theta_{m,\ep}^+)\sin\xi)\bigg\rangle,
		\end{align*}
		\begin{align*}
			\mathcal F_n^-(\xi)&=\bigg\langle \sum\limits_{l\neq n}^k\kappa_{l}^-\nabla G(\boldsymbol z_{n,\ep}^-,\boldsymbol z^-_{l,\ep})+\kappa_n^- \nabla H(\boldsymbol z_{n,\ep}^-,\boldsymbol z^-_{n,\ep})-\sum\limits_{m=1}^j\kappa_m^+\nabla G(\boldsymbol z_{n,\ep}^-,\boldsymbol z^+_{m,\ep})\\
			&\quad-(W\sin\theta_{m,\ep}^+,0),\quad (\ep\cos\xi,\ep\sin^{-1}(\ep y_1+\theta_{n,\ep}^-)\sin\xi)\bigg\rangle.
		\end{align*}
		Moreover, it holds
		\begin{equation}\label{}
			\left|\tilde{\mathbf\Gamma}_{l,\phi_1}^\pm-\tilde{\mathbf\Gamma}_{l,\phi_2}^\pm\right|=\left(\frac{\kappa}{\kappa_{l}^\pm w'_1(\tau)}+O(\varepsilon)\right)\|\tilde\phi_{l,1}^\pm-\tilde\phi_{l,2}^\pm\|_{L^\infty(B_L(\boldsymbol 0))}
		\end{equation}
		for two functions $\phi_1,\phi_2$ satisfying \eqref{2-10}.
	\end{lemma}
	
	Then we can verify the existence by the contraction mapping theorem.
	
	\begin{lemma}\label{lem3-3}
		There exists a small $\varepsilon_0>0$ such that for any $\varepsilon\in(0,\varepsilon_0]$, there is a unique solution $\phi_\varepsilon\in E_\varepsilon$ to \eqref{2-4}. Moreover  $\phi_\varepsilon$ satisfies
		\begin{equation}\label{3-6}
			\|\phi_\varepsilon\|_*+\varepsilon^{1-\frac{2}{p}}\|\nabla\phi_\varepsilon\|_{L^p\left( \left(\cup_{m=1}^jB_{L\varepsilon}(\boldsymbol z_{m,\varepsilon}^+)\right)\cup \left(\cup_{n=1}^kB_{L\varepsilon}(\boldsymbol z_{n,\varepsilon}^-)\right)\right)}= O(\varepsilon)
		\end{equation}
		for $p\in(2,+\infty]$.
	\end{lemma}
	\begin{proof}
		Denote $\mathcal G_\varepsilon:=\mathcal T_\varepsilon R_\varepsilon$, and a neighborhood of origin in $E_\varepsilon$ as
		\begin{equation*}
			\mathcal B_\varepsilon:=E_\varepsilon\cap \left\{\phi \ | \ \|\phi\|_*+\varepsilon^{1-\frac{2}{p}}\|\nabla\phi\|_{L^p\left( \left(\cup_{m=1}^jB_{L\varepsilon}(\boldsymbol z_{m,\varepsilon}^+)\right)\cup \left(\cup_{n=1}^kB_{L\varepsilon}(\boldsymbol z_{n,\varepsilon}^-)\right)\right)}\le C\varepsilon, \ p\in(2,\infty]\right\}
		\end{equation*}
		with $C$ a large positive constant. Then for $\phi\in\mathcal B_\varepsilon$, we have
		\begin{align*}
			\|\mathcal T_\varepsilon N_\varepsilon(\phi)\|_*+\varepsilon^{1-\frac{2}{p}}\|\nabla\mathcal T_\varepsilon N_\varepsilon(\phi)&\|_{L^p\left( \left(\cup_{m=1}^jB_{L\varepsilon}(\boldsymbol z_{m,\varepsilon}^+)\right)\cup \left(\cup_{n=1}^kB_{L\varepsilon}(\boldsymbol z_{n,\varepsilon}^-)\right)\right)}\\
			&\le c_0\varepsilon^{1-\frac{2}{p}}\|N_\varepsilon(\phi) \|_{W^{-1,p}\left( \left(\cup_{m=1}^jB_{L\varepsilon}(\boldsymbol z_{m,\varepsilon}^+)\right)\cup \left(\cup_{n=1}^kB_{L\varepsilon}(\boldsymbol z_{n,\varepsilon}^-)\right)\right)}.
		\end{align*}
		
		\smallskip
		
		 To verify the continuity of $\mathcal G_\ep$, we only consider $N_\varepsilon(\phi_\varepsilon)$ restricted in $B_\delta(\boldsymbol z_{m,\ep}^+)$ for simplicity. Let
		 \begin{equation*}
		 	d_\ep=\sup_{\theta\in [0,2\pi)} |t_\ep(\theta)|
		 \end{equation*}   
		 with $t_\ep$ is the parametrization of the boundary $\partial \Om_{m,\ep}^+$ defined by Lemma \ref{lem3-2}. For each $\zeta(\boldsymbol y)\in C_0^\infty(B_{L}(\boldsymbol 0))$, by the estimate for $\Omega_{m,\varepsilon}^+$ in lemma \ref{lem3-2}, direct calculation yields
		\begin{align*}
			&\quad\int_{B_L(\boldsymbol 0)}\left[\left(\tilde\psi_{m,\varepsilon}^++W\cos(\ep y_1+\theta_{m,\varepsilon}^+)-\mu_{m,\varepsilon}^+\right)_+-\left(\tilde V_{m,\varepsilon}^+-\frac{\kappa_m^+}{2\pi}\ln\frac{1}{\varepsilon}\right)_+\right]\zeta d\boldsymbol y-\int_{B_\tau(\boldsymbol 0)}\phi \zeta d\boldsymbol y\\
			&=\int_{B_{\tau+d_\ep}(\boldsymbol 0)\setminus B_{\tau-d_\ep}(\boldsymbol 0)}\left[\phi_\ep+\mathcal F_m^++O(\ep^2)\right]\zeta d\boldsymbol y\\
			&=O(\varepsilon)\cdot\|\zeta\|_{W^{1,p'}(B_{L}(\boldsymbol 0)}.
		\end{align*}
		Hence it holds
		\begin{equation*}
			\varepsilon^{1-\frac{2}{p}}\|N_\varepsilon(\phi)\|_{W^{-1,p}\left( \left(\cup_{m=1}^jB_{L\varepsilon}(\boldsymbol z_{m,\varepsilon}^+)\right)\cup \left(\cup_{n=1}^kB_{L\varepsilon}(\boldsymbol z_{n,\varepsilon}^-)\right)\right)}=O(\varepsilon),
		\end{equation*}
		and
		\begin{equation*}
			\|\mathcal T_\varepsilon N_\varepsilon(\phi)\|_*+\varepsilon^{1-\frac{2}{p}}\|\nabla\mathcal T_\varepsilon N_\varepsilon(\phi)\|_{L^p\left( \left(\cup_{m=1}^jB_{L\varepsilon}(\boldsymbol z_{m,\varepsilon}^+)\right)\cup \left(\cup_{n=1}^kB_{L\varepsilon}(\boldsymbol z_{n,\varepsilon}^-)\right)\right)}=O(\varepsilon)
		\end{equation*}
		for $p\in(2,+\infty]$, and the first step is finished.
		
		\smallskip
		
		The next step is to verify that $\mathcal G_\varepsilon$ is a contraction mapping under the norm
		\begin{equation*}
			\|\cdot\|_{\mathcal G_\varepsilon}=\|\cdot\|_*+\varepsilon^{1-\frac{2}{p}}\|\cdot\|_{W^{1,p}\left( \left(\cup_{m=1}^jB_{L\varepsilon}(\boldsymbol z_{m,\varepsilon}^+)\right)\cup \left(\cup_{n=1}^kB_{L\varepsilon}(\boldsymbol z_{n,\varepsilon}^-)\right)\right)}, \quad p\in(2,+\infty].
		\end{equation*}
		Let $\phi_1$ and $\phi_2$ be two functions in $\mathcal B_\varepsilon$, we have
		\begin{equation*}
			\|\mathcal G_\varepsilon\phi_1-\mathcal G_\varepsilon\phi_2\|_{\mathcal G_\varepsilon}\le C\varepsilon^{1-\frac{2}{p}}\|N_\varepsilon^+(\phi_1)-N_\varepsilon^+(\phi_2) \|_{W^{-1,p}\left( \left(\cup_{m=1}^jB_{L\varepsilon}(\boldsymbol z_{m,\varepsilon}^+)\right)\cup \left(\cup_{n=1}^kB_{L\varepsilon}(\boldsymbol z_{n,\varepsilon}^-)\right)\right)},
		\end{equation*}
		where
		\begin{align*}
			&\quad N_\varepsilon(\phi_1)-N_\varepsilon(\phi_2)\\
			&=\frac{1}{\varepsilon^2}\sum_{m=1}^j\boldsymbol 1_{B_\delta(\boldsymbol z_m^+)}\big[(\Psi_\varepsilon+\phi_1+W\cos\theta>\mu_{m,\varepsilon}^+)_+-(\Psi_\varepsilon+\phi_2+W\cos\theta>\mu_{m,\varepsilon}^+)_+\big]\\
			&\quad -\frac{1}{\varepsilon^2}\sum_{m=1}^j\boldsymbol 1_{B_{\tau\ep}(\boldsymbol z_{m,\ep}^+)}\cdot(\phi_1-\phi_2)\\
			&\quad-\frac{1}{\varepsilon^2}\sum_{n=1}^k\boldsymbol 1_{B_\delta(\boldsymbol z_n^-)}\big[(-\Psi_\varepsilon-\phi_1-W\cos\theta>\mu_{n,\varepsilon}^-)_+-(-\Psi_\varepsilon-\phi_2-W\cos\theta>\mu_{n,\varepsilon}^-)_+\big]\\
			&\quad -\frac{1}{\varepsilon^2}\sum_{n=1}^k\boldsymbol 1_{B_{\tau\ep}(\boldsymbol z_{n,\ep}^-)}(\phi_1-\phi_2).
		\end{align*}
		Then for each $\zeta\in C_0^\infty(B_L(\boldsymbol 0))$, it holds
		\begin{align*}
			&\quad\,\, \int_{B_L(\boldsymbol 0)}\left[(\tilde\Psi_{m,\varepsilon}^++\tilde\phi_{m,1}^++W\cos(\ep y_1+\theta_{m,\ep}^+)>\mu_{m,\varepsilon}^+)_+\right.\\
			&\left.\quad\quad-(\tilde\Psi_{m,\varepsilon}^++\tilde\phi_{m,2}^++W\cos(\ep y_1+\theta_{m,\ep}^+)>\mu_{m,\varepsilon}^+)_+\right]\zeta d\boldsymbol y\\
			&=\int_{B_\tau(\boldsymbol 0)} (\phi_1-\phi_2)\zeta d\boldsymbol y\\
			&\quad\quad+\int_{B_{1+t_{\ep,\phi_1}}(\boldsymbol 0)\setminus B_{1+t_{\ep,\phi_2}}(\boldsymbol 0)}\left(\tilde\Psi_{m,\varepsilon}^++W\cos(\ep y_1+\theta_{m,\varepsilon}^+)-\mu_{m,\varepsilon}^+\right)_+\zeta d\boldsymbol y\\
			&\quad\quad +\int_{B_{1+t_{\ep,\phi_1}}(\boldsymbol 0)\setminus B_\tau(\boldsymbol 0)}\phi_1\zeta d\boldsymbol y-\int_{B_{1+t_{\ep,\phi_2}}(\boldsymbol 0)\setminus B_\tau(\boldsymbol 0)}\phi_2\zeta d\boldsymbol y\\
			&=\int_{B_\tau(\boldsymbol 0)} (\phi_1-\phi_2)\zeta d\boldsymbol y+o_\ep(1)\cdot\|\tilde\phi_1-\tilde\phi_2\|_{L^\infty(B_L(\boldsymbol 0))}\|\zeta\|_{W^{1,p'}(B_L(\boldsymbol 0))}.
		\end{align*}
		Thus we have
		\begin{align*}
			&\quad\,\, \int_{B_L(\boldsymbol 0)}\left[(\tilde\Psi_{m,\varepsilon}^++\tilde\phi_{m,1}^++W\cos(\ep y_1+\theta_{m,\ep}^+)>\mu_{m,\varepsilon}^+)_+\right.\\
			&\left.\quad\quad-(\tilde\Psi_{m,\varepsilon}^++\tilde\phi_{m,2}^++W\cos(\ep y_1+\theta_{m,\ep}^+)>\mu_{m,\varepsilon}^+)_+\right]\zeta d\boldsymbol y-\int_{B_\tau(\boldsymbol 0)}(\phi_1-\phi_2)\zeta d\boldsymbol y\\
			&=o_\ep(1)\|\tilde\phi_1-\tilde\phi_2\|_{L^\infty(B_L(\boldsymbol 0))}||\zeta ||_{W^{1,p'}(B_L(\boldsymbol 0))}.
		\end{align*}
		By dealing with other vortices using the same procedure, we conclude that
		\begin{equation*}
			\varepsilon^{1-\frac{2}{p}}\|N_\varepsilon(\phi_1)-N_\varepsilon(\phi_2) \|_{W^{-1,p}\left( \left(\cup_{m=1}^jB_{L\varepsilon}(\boldsymbol z_{m,\varepsilon}^+)\right)\cup \left(\cup_{n=1}^kB_{L\varepsilon}(\boldsymbol z_{n,\varepsilon}^-)\right)\right)}=o_\varepsilon(1)\cdot \|\phi_1-\phi_2\|_{\mathcal G_\varepsilon},
		\end{equation*}
		and
		\begin{equation*}
			\|\mathcal G_\varepsilon\phi_1-\mathcal G_\varepsilon\phi_2\|_{\mathcal G_\varepsilon}=o_\varepsilon(1) \|\phi_1-\phi_2\|_{\mathcal G_\varepsilon}.
		\end{equation*}
		Hence we have shown that $\mathcal G_\varepsilon$ is a contraction map from $\mathcal B_\varepsilon$ into itself, and obtain a unique $\phi_\varepsilon\in \mathcal B_\varepsilon$ such that $\phi_\varepsilon=\mathcal G_\varepsilon\phi_\varepsilon$, which satisfies \eqref{3-6}. 
	\end{proof}
	
	To ensure $\psi_\ep=\Psi+\phi_\ep$ a solution to \eqref{1-6} with $\gamma=1$, we need to choose the vortex centers $$(\boldsymbol z_{1,\varepsilon}^+,\cdots,\boldsymbol z_{j,\varepsilon}^+, \boldsymbol z_{1,\varepsilon}^-,\cdots, \boldsymbol z_{k,\varepsilon}^-)$$
	so that $\mathbf \Lambda=\boldsymbol 0$. By multiplying \eqref{3-2} with $X_{l,\varepsilon}^\pm$, $Y_{i,\varepsilon}^\pm$ and integrating on $\mathbb S^2$,  we have 
	\begin{footnotesize}
		\begin{align*}
			&\int_{\mathbb S^2}\left[(-\Delta_{\mathbb S^2})\psi_\varepsilon-\frac{1}{\varepsilon^2}\sum_{m=1}^j\boldsymbol1_{B_\delta(\boldsymbol z_m^+)}(\psi_\varepsilon+W\cos\theta-\mu_{m,\varepsilon}^+)_++\frac{1}{\varepsilon^2}\sum_{n=1}^k\boldsymbol1_{B_\delta(\boldsymbol z_n^-)}(-\psi_\varepsilon-W\cos\theta-\mu_{n,\varepsilon}^-)_+\right]X_{m,\varepsilon}^+ d\boldsymbol\sigma\\
			&=a_m^+\int_{\mathbb S^2} X_{m,\varepsilon}^+(-\Delta_{\mathbb S^2})X_{m,\varepsilon}^+ d\boldsymbol \sigma,
		\end{align*}
		\begin{align*}
			&\int_{\mathbb S^2}\left[(-\Delta_{\mathbb S^2})\psi_\varepsilon-\frac{1}{\varepsilon^2}\sum_{m=1}^j\boldsymbol1_{B_\delta(\boldsymbol z_m^+)}(\psi_\varepsilon+W\cos\theta-\mu_{m,\varepsilon}^+)_++\frac{1}{\varepsilon^2}\sum_{n=1}^k\boldsymbol1_{B_\delta(\boldsymbol z_n^-)}(-\psi_\varepsilon-W\cos\theta-\mu_{n,\varepsilon}^-)_+\right]Y_{m,\varepsilon}^+ d\boldsymbol\sigma\\
			&=b_m^+\int_{\mathbb S^2} Y_{m,\varepsilon}^+(-\Delta_{\mathbb S^2})Y_{m,\varepsilon}^+ d\boldsymbol \sigma,
		\end{align*}
		\begin{align*}
			&\int_{\mathbb S^2}\left[(-\Delta_{\mathbb S^2})\psi_\varepsilon-\frac{1}{\varepsilon^2}\sum_{m=1}^j\boldsymbol1_{B_\delta(\boldsymbol z_m^+)}(\psi_\varepsilon+W\cos\theta-\mu_{m,\varepsilon}^+)_++\frac{1}{\varepsilon^2}\sum_{n=1}^k\boldsymbol1_{B_\delta(\boldsymbol z_n^-)}(-\psi_\varepsilon-W\cos\theta-\mu_{n,\varepsilon}^-)_+\right]X_{n,\varepsilon}^- d\boldsymbol\sigma\\
			&=a_n^-\int_{\mathbb S^2} X_{n,\varepsilon}^-(-\Delta_{\mathbb S^2})X_{n,\varepsilon}^- d\boldsymbol \sigma,
		\end{align*}
		\begin{align*}
			&\int_{\mathbb S^2}\left[(-\Delta_{\mathbb S^2})\psi_\varepsilon-\frac{1}{\varepsilon^2}\sum_{m=1}^j\boldsymbol1_{B_\delta(\boldsymbol z_m^+)}(\psi_\varepsilon+W\cos\theta-\mu_{m,\varepsilon}^+)_++\frac{1}{\varepsilon^2}\sum_{n=1}^k\boldsymbol1_{B_\delta(\boldsymbol z_n^-)}(-\psi_\varepsilon-W\cos\theta-\mu_{n,\varepsilon}^-)_+\right]Y_{n,\varepsilon}^- d\boldsymbol\sigma\\
			&=b_n^-\int_{\mathbb S^2} Y_{n,\varepsilon}^-(-\Delta_{\mathbb S^2})Y_{n,\varepsilon}^- d\boldsymbol \sigma.
		\end{align*}
	\end{footnotesize}
	
	Notice that lemma \ref{lem2-6} also holds for $\gamma=1$. Since $(\boldsymbol z_1^+,\cdots,\boldsymbol z_j^+, \boldsymbol z_1^-,\cdots, \boldsymbol z_k^-)$ is a nondegenerate critical point of Kirchhoff--Routh function $\mathcal K_{k+j}$ in \eqref{1-5}, we claim that there exists a proper location series
	\begin{equation*}
		\left(\boldsymbol z_{1,\varepsilon}^+,\cdots,\boldsymbol z_{j,\varepsilon}^+, \boldsymbol z_{1,\varepsilon}^-,\cdots, \boldsymbol z_{k,\varepsilon}^-\right)= \left(\boldsymbol z_1^+,\cdots,\boldsymbol z_j^+, \boldsymbol z_1^-,\cdots, \boldsymbol z_k^-\right)+o_\varepsilon(1)
	\end{equation*}
	such that $\mathbf \Lambda=\boldsymbol 0$. The convergence for $\psi_\ep$ and $\omega_\ep$ is then a direct result of our construction, and the estimates for vortex sets $\Omega_{m,\varepsilon}^+$ and $\Omega_{n,\varepsilon}^-$ are from Lemma \ref{lem3-3}. Moreover, for $\gamma=1$, the vorticity function $\omega_\ep$ is also $C^1$ smooth by the regularity theory of the elliptic operator. Thus we can deduce that $\ep^2\|\phi_\ep\|_{C^2(B_L(\boldsymbol 0))}=O(\ep)$ by bootstrap, and
	$$\|t_{\varepsilon}\|_{C^2[0,2\pi)}=\frac{\kappa \| \tilde\phi_l^\pm+\mathcal F_l^\pm\|_{C^2(B_L(\boldsymbol 0))}}{\kappa_{l}^\pm w_1'(\tau)+O(\varepsilon)}=O(\varepsilon)$$
	by implicit function theorem. The curvature of $\partial\Omega_{l,\ep}^\pm$ is
	$$\boldsymbol \kappa_{l,\ep}^\pm= \frac{(1+t_\ep)^2+2t_\ep'^2-(1+t_\ep)t_\ep''}{\ep[(1+t_\ep)^2+t_\ep''^2]^{\frac{3}{2}}}=\frac{1+O(\ep)}{\ep}>0.$$
	Hence the proof is complete for Theorem \ref{thm2} on the case $\gamma=1$.
	
	\bigskip
	
	\noindent{\bf Conflict of interest statement:} On behalf of all authors, the corresponding author states that there is no conflict of interest.
	
	\bigskip
	
	\noindent{\bf Data available statement:} Our manuscript has no associated data.

	\bigskip
	 
    \noindent{\bf Acknowledgement:}  T.S. is partially supported by JSPS KAKENHI, Grant No. 23H00086. C.Z. is supported by by NNSF of China (Grant 12301142 and 12371212), CPSF (Grant 2022M722286), NSFSC (Grant 2024NSFSC1341), and International Visiting Program for Excellent Young Scholars of SCU.

\end{document}